\documentclass[10pt]{amsart}
\usepackage{amsmath}
\usepackage{amsfonts}
\usepackage{amsthm}
\usepackage{amssymb}
\usepackage{mathrsfs}
\usepackage{latexsym}
\usepackage{verbatim}
\usepackage{dsfont}
\usepackage{stmaryrd}
\usepackage{graphics}
\usepackage{a4wide}
\usepackage[english,french]{babel}

\usepackage{amsmath}

\usepackage{amssymb}

\usepackage{amscd} \usepackage{tabularx}
\usepackage[all,cmtip,line]{xy}

\newcommand{\ra}{\rightarrow}

\newcommand{\into}{\hookrightarrow}

\newlength{\ownl}

\newcommand{\Ext}{{\operatorname{Ext}\,}}

\newcommand{\Frob}{{\operatorname{Frob}}}

\newcommand{\Gal}{{\operatorname{Gal}\,}}

\newcommand{\Hom}{{\operatorname{Hom}\,}}

\newcommand{\Ind}{{\operatorname{Ind}\,}}

\newcommand{\ad}{{\operatorname{ad}\,}}

\newcommand{\tr}{{\operatorname{tr}\,}}

\newcommand{\semis}{{\operatorname{ss}}}

\newcommand{\A}{{\mathbb{A}}}

\newcommand{\C}{{\mathbb{C}}}

\newcommand{\F}{{\mathbb{F}}}

\newcommand{\Q}{{\mathbb{Q}}}
\newcommand{\R}{{\mathbb{R}}}

\newcommand{\Z}{{\mathbb{Z}}}

\newcommand{\gog}{{\mathfrak{g}}}
\newcommand{\gh}{{\mathfrak{h}}}

\newcommand{\gz}{{\mathfrak{z}}}

\newcommand{\Mbar}{\overline{{M}}}

\newcommand{\Qbar}{{\overline{\Q}}}

\newcommand{\kbar}{{\overline{k}}}

\def\RCS$#1: #2 ${\expandafter\def\csname RCS#1\endcsname{#2}}
\RCS$Revision: 86 $
\RCS$Date: 2016-09-14 15:26:10 +0100 (Wed, 14 Sep 2016) $

\newcommand{\cL}{\mathcal{L}}

\newcommand{\chibar}{\bar{\chi}}
\newcommand{\phibar}{\bar{\phi}}

\DeclareMathOperator{\GO}{GO}
\DeclareMathOperator{\SO}{SO}

\newcommand{\rhobar}{\overline{\rho}} 
\newcommand{\rbar}{\bar{r}}
\newcommand{\sbar}{\bar{s}}
\newcommand{\tbar}{\bar{t}}

\newcommand{\GL}{\operatorname{GL}}
\newcommand{\GSp}{\operatorname{GSp}}
\newcommand{\Sp}{\operatorname{Sp}}
\newcommand{\PGL}{\operatorname{PGL}}
\newcommand{\HT}{\operatorname{HT}}

\renewcommand{\fg}{\mathfrak{g}}

\newcommand{\cA}{\mathcal{A}}

\newcommand{\cO}{\mathcal{O}}

\newcommand{\cR}{\mathcal{R}}
\newcommand{\cS}{\mathcal{S}}

\newcommand{\Qlbar}{\overline{\Q}_{l}}

\newcommand{\Fl}{{\F_l}}
\newcommand{\Flbar}{\overline{\F}_l}

\newcommand{\SL}{\operatorname{SL}}
\newcommand{\PSL}{\operatorname{PSL}}

\newcommand{\Sym}{\operatorname{Sym}}

\def\abar{\overline{a}}
\def\bbar{\overline{b}}

\def\OL{\mathcal{O}}

\newcommand\Sel{\mathcal{L}}

\usepackage{hyperref}

\usepackage{amsthm}

\newtheorem{thm}{Theorem}[section]
\newtheorem{corollary}[thm]{Corollary}
\newtheorem{cor}[thm]{Corollary}
 \newtheorem{lemma}[thm]{Lemma}
\newtheorem{lem}[thm]{Lemma} \newtheorem{prop}[thm]{Proposition}
 \theoremstyle{definition}
 \theoremstyle{definition}
\newtheorem{defn}[thm]{Definition} \theoremstyle{remark}
\newtheorem{df}[thm]{Definition}

\numberwithin{equation}{section}

\theoremstyle{definition}

\newcommand\sss{\mathrm{ss}}

\newcommand\slf{\mathfrak{sl}}
\renewcommand\sl{\mathfrak{sl}}
\newcommand\spf{\mathfrak{sp}}
\newcommand\sof{\mathfrak{so}}
\newcommand\so{\mathfrak{so}}
\renewcommand\sp{\mathfrak{sp}}

\begin{document}
\selectlanguage{english}
\title[Irreducibility of automorphic Galois representations]{Irreducibility of automorphic Galois
  representations of $\GL(n)$, $n$ at most $5$.}

\author{Frank Calegari} \email{fcale@math.northwestern.edu} \address{Department of Mathematics,
Northwestern University} 
\author{Toby Gee} \email{gee@math.northwestern.edu} \address{Department of
  Mathematics, Northwestern University}  \thanks{The first author was partially supported
  by a Sloan foundation grant and 
NSF Career Grant DMS-0846285. The second author was partially supported
  by NSF grant DMS-0841491}  

\begin{abstract}
  Let $\pi$ be a regular, algebraic, essentially self-dual cuspidal
  automorphic representation of $\GL_n(\A_F)$, where $F$ is a totally
  real field and $n$ is at most $5$.  We show that for all primes $l$,
  the $l$-adic Galois representations associated to $\pi$ are
  irreducible, and for all but finitely many primes $l$, the mod $l$
  Galois representations associated to $\pi$ are also irreducible. We
  also show that the Lie algebras of the Zariski closures of the $l$-adic representations are
  independent of $l$.
\end{abstract}
\maketitle
\selectlanguage{english}

\section*{Erratum}

{\large
{\bf WARNING:\rm}
 There is an error in (the statement of) Theorem 4.7, namely, it does
 not apply to the group $H = \mathrm{GO}(V) \subset \mathrm{GL}(V)$
 when $\mathrm{dim}(V) = 5$. This invalidates the proof of Theorem 4.1
 both for $n = 5$ and $n = 4$. For $n = 4$, our methods seem
 inadequate to distinguish between the possibility that the Lie
 algebra of the image is $\mathfrak{sl}_2 \times \mathfrak{sl}_2$ or
 $\mathfrak{sp}_4$, because these groups have $4$-dimensional
 representations with the same formal character. The most general
results that we know of in the case $n=4$ are
those of~\cite{MR3156860}.

The results of the first three sections remain valid, as do Lemma
4.3 and Corollary 4.4, and the results of section~5. (The statement of Proposition~4.5 is missing the symmetric square
 of the standard representation of~$\mathfrak{sl}_3$.)

We would like to thank Susan Xia for bringing this mistake to our attention.
}

 \section{Introduction} \subsection{}It is a folklore conjecture that
 the Galois representations (conjecturally) associated to algebraic
 cuspidal automorphic representations of $\GL_n(\A_F)$ over a number
 field $F$ are all irreducible. In general, rather little is known in
 this direction. Ribet (\cite{MR0453647}) proved this result for
 classical modular forms, and his proof extends to the case of Hilbert
 modular forms (\cite{MR1363502}). The result was proved for essentially self-dual
 representations of $\GL_3(\A_F)$, $F$ totally real,
 in~\cite{MR1155236}. 
 In~\cite{0810.2203}, Dieulefait and Vila proved big image results
for a compatible family arising from a rank
four pure motive $M$ over $\Q$ with Hodge-Tate
weights $(0,1,2,3)$, coefficients in
a quadratic field $K$, and  certain other
supplementary hypotheses (see also~\cite{MR2472504}).
   In a 2009 preprint
 (\cite{ramakrishnan2009irreducibility}), Ramakrishnan proves
 irreducibility of the associated $l$-adic representations for
 essentially self-dual representations of $\GL_4(\A_{\Q})$ for
 sufficiently large $l$; his argument also applies without the
 assumption of self-duality \emph{assuming} the existence of the
 corresponding Galois representations.

Until recently, very little was known in the general case. It is
sometimes the case that the Galois representations can be proved to be
irreducible for purely local reasons; if the automorphic
representation is square-integrable at some finite place, then it is a
consequence of the expected local-global compatibility that the
corresponding local Galois representation is indecomposable, which
implies that the global Galois representation, being semisimple, is 
irreducible. In~\cite{ty}, this observation was used to prove the
irreducibility of the Galois representations considered in~\cite{ht}
whenever the square-integrability hypothesis holds.

One reason to suppose that the Galois representations should be
irreducible is that if the Fontaine--Mazur--Langlands conjecture
holds, then (by a standard $L$-function argument) the reducibility of
an $l$-adic Galois representation associated to an automorphic
representation would show that the automorphic representation could
not be cuspidal. In fact, it is enough to know that any geometric
Galois representation is potentially automorphic, as this $L$-function
argument is compatible with the usual arguments involving Brauer's
theorem. This observation was exploited in~\cite{blggt} to prove that
if $K$ is an imaginary CM field and $\pi$ is a regular, algebraic,
essentially self-dual cuspidal automorphic representation of
$\GL_n(\A_K)$ which has extremely regular weight (a notion defined in
\cite{blggt}), then for a density one set of primes $l$, the $l$-adic
Galois representations associated to $\pi$ are irreducible.  While
this theorem is useful in practice, the condition that the weight of
$\pi$ be extremely regular is sometimes too restrictive. For example,
it is never satisfied by the base change of an automorphic
representation over a totally real field if $n>3$. In the present
paper, we begin by extending the result of~\cite{blggt} to the case of
totally real fields if $n\le 5$, with no assumption on the weight of
$\pi$. Just as in~\cite{blggt}, we use potential automorphy theorems
and the $L$-function argument mentioned above. The key difficulty with
applying the potential automorphy theorems available to us is to show
that any hypothetical summand of the Galois representation to be
proved irreducible is both essentially self-dual and odd. It is here
that the arguments of~\cite{blggt} make use of the condition of
extreme regularity, which is not available to us. Instead, we observe
that if $n\le 5$ then any constituent of dimension at least $3$ must
be essentially self-dual for dimension reasons, and is then odd by the
main result of~\cite{belchen}. One dimensional summands are trivial to
deal with, which leaves us with the problem of dealing with
2-dimensional summands. However, any two-dimensional representation is
essentially self-dual, so we need only show that we cannot have even
two-dimensional constituents, at least outside of a set of places of
density zero. To do this, we use a variant of the arguments
of~\cite{frank}, together with an argument using class field theory to
show that there cannot be too many residually dihedral
representations.

The arguments outlined so far suffice to prove the result for a set of
primes of density $1$.  In order to extend our result to all primes,
we make use of a group-theoretic argument (in combination with our
density $1$ result) to show that the characteristic polynomials of the
images of the Frobenius elements can only be divisible by the
characteristic polynomials of a global character in certain special
cases, which rules out the possibility of any of the Galois
representations having a one-dimensional summand. We use the same
argument to show that it is not possible for any of the
representations to have a dihedral summand. We make use of the
self-duality of the Galois representations we consider to reduce to
these possibilities and so obtain the result.

In order to extend this argument to the characteristic $l$ representations, we
show using class field theory that if infinitely many of the
characteristic $l$ representations have a one-dimensional summand,
then the characteristic polynomials of the images of the Frobenius
elements are divisible by the characteristic polynomials of a global
character, which reduces us to the cases above. We prove a similar
result for dihedral representations. This quickly reduces us to one
special case, that of an irreducible $4$-dimensional subrepresentation
which when reduced mod $l$ splits up as a sum of two irreducible
$2$-dimensional representations.  In this case, we are able to exploit
the connection between $\GO_4$ and $\GL_2\times\GL_2$ to reach a
contradiction.

Using similar arguments, we are also able to show that the Lie
algebras of the Zariski closures of the images of the $l$-adic
representations are independent of $l$. Following
\cite{ramakrishnan2009irreducibility}, we additionally extend our
analysis to the Galois representations associated to regular algebraic
cuspidal automorphic representations of $\GL_3$ or $\GL_4$ over a
totally real field which are not assumed to be essentially self-dual,
under the hypothesis that the Galois representations exist.

One may naturally ask whether these methods can be generalized to $n
\ge 6$; we explain why this might be difficult.  Suppose that $\pi$ is
a regular algebraic cuspidal automorphic representation of
$\GL_3(\A_{\Q})$ which is not essentially self-dual (they
exist!). Then, conjecturally, there should exist a compatible system
of three dimensional Galois representations $\cR =
\{r_{\lambda}(\pi)\}$ of $G_{\Q}$. For a sufficiently large integer
$n$, the compatible system $\cR \oplus (\epsilon^n \otimes
\cR^{\vee})$ is a six dimensional compatible system of essentially
self-dual regular Galois representations.  Our method for ruling out
that this (completely reducible) compatible system is associated to a
regular, algebraic essentially self-dual cuspidal automorphic
representation $\Pi$ of $\GL_6(\A_{\Q})$ would consist of recognizing
it as an isobaric sum $\pi \boxplus (\pi^{\vee} \otimes | \cdot |^n)$
by proving the (potential) automorphy of a non-essentially self-dual
representation $r_{\lambda}(\pi): G_{\Q} \rightarrow \GL_3(\Qbar_{l})$
for some prime $l$.  However, such automorphy results are out of reach
at present.

We would like to thank Florian Herzig for helpful conversations about representation theory,
and Dinakar Ramakrishan for making available to us the preprint~\cite{ramakrishnan2009irreducibility}.
We would also especially like to thank Robert Guralnick, who (together with Malle in~\cite{guralnickmalle}) answered a difficult
problem of the first author in the modular representation theory of finite groups. 
Even though (due to  subsequent simplifications)
we did not end up using this result,  the mere knowledge of its veracity was helpful psychologically
in the construction of several of our arguments. We would also like to
thank Tom Barnet-Lamb, Kevin Buzzard, Luis Dieulefait, Matthew Emerton
and Florian Herzig for their helpful comments on an earlier draft of
this paper, and the referee for their helpful comments and corrections.

\section{Preliminaries}\subsection{}
We recall some notions from~\cite{blggt}. Let $F$ be a totally real
field. By a {\em RAESDC} (regular, algebraic, essentially self-dual,
cuspidal) automorphic representation of $\GL_n(\A_{F})$, we mean a pair
$(\pi,\chi)$ where
\begin{itemize}
\item $\pi$ is a cuspidal automorphic representation of $\GL_n(\A_F)$
  such that $\pi_\infty$ has the same infinitesimal character as some
  irreducible algebraic representation of the restriction of scalars
  from $F$ to $\Q$ of $GL_n$,
\item $\chi:\A_{F}^\times/F^\times \ra \C^\times$ is a continuous
  character such that $\chi_v(-1)$ is independent of $v|\infty$,
\item and $\pi \cong \pi^\vee \otimes (\chi \circ \det)$.
\end{itemize}

If $\Omega$ is an algebraically closed field of characteristic $0$, we 
write $(\Z^n)^{\Hom(F,\Omega),+}$ for the set of $a=(a_{\tau,i}) \in (\Z^n)^{\Hom(F,\Omega)}$ satisfying
\[ a_{\tau,1} \geq \dots \geq a_{\tau,n}. \] If $a \in (\Z^n)^{\Hom(F,\C),+}$, let $\Xi_a$ denote the irreducible
algebraic representation of $GL_n^{\Hom(F,\C)}$ which is the tensor
product over $\tau$ of the irreducible representations of $GL_n$ with
highest weights $a_\tau=(a_{\tau,i})_{1\le i\le n}$. We say that a RAESDC automorphic
representation $(\pi,\chi)$ of $GL_n(\A_F)$ has {\em weight $a$} if
$\pi_\infty$ has the same infinitesimal character as $\Xi_a^\vee$
(this is necessarily the case for some unique $a$). There is
necessarily an integer $w$ such that \[ a_{\tau,i}+a_{\tau,
  n+1-i}=w \] for all $\tau$, $i$ (cf. section 2.1 of~\cite{blggt}).

We refer the reader to section 5.1 of~\cite{blggt} for the definition
of a compatible system of Galois representations, and for various
attendant definitions. If $(\pi,\chi)$ is a RAESDC automorphic
representation of $\GL_n(\A_F)$, then there is a number field $M$
containing the images of all embeddings $F\into \Mbar$ and weakly compatible systems of Galois
representations \[r_\lambda(\pi):G_F\to\GL_n(\Mbar_\lambda)\] and
\[r_\lambda(\chi):G_F\to\Mbar_\lambda^\times\] as $\lambda$ ranges over
the finite places of $M$ (cf. the last paragraph of section 5.1 of
\cite{blggt}). Suppose that $\pi$ has weight
$a\in(\Z^n)^{\Hom(F,\C),+}$, and regard each element of $\Hom(F,\C)$
as an element of $\Hom(F,\Mbar)$. Then:
\begin{itemize}
\item if $S$ is the finite
set of finite places $v$ of $F$ at which $\pi_v$ is ramified, then
$r_\lambda(\pi)$ and $r_\lambda(\chi)$ are unramified unless $v\in S$
or $v|l$;
\item $r_\lambda(\pi)\cong r_\lambda(\pi)^\vee\otimes\epsilon^{1-n}r_\lambda(\chi)$;
\item if $v|l$ then $r_\lambda(\pi)|_{G_{F_v}}$ and
  $r_\lambda(\chi)|_{G_{F_v}}$ are de Rham. If furthermore $v\notin S$
  then $r_\lambda(\pi)|_{G_{F_v}}$ and $r_\lambda(\chi)|_{G_{F_v}}$
  are crystalline;
\item for each $\tau:F\into\Mbar$ and any $\Mbar\into\Mbar_\lambda$
  over $M$, the set $\HT_\tau(r_\lambda(\pi))$ of $\tau$-Hodge-Tate
  weights of $r_\lambda(\pi)$ is equal to \[\{a_{\tau,1}+(n-1),a_{\tau,2}+(n-2),\dots,a_{\tau,n}\}.\]
\end{itemize}
In arguments it will occasionally be useful to replace $M$ with a
finite extension, in order to compare two different compatible
systems; we will do this without comment.

While we will not make explicit use of this fact, to orient the reader
we remark that if $v\notin S,v\nmid l$ is a finite place of $F$, and $\Frob_v$
is a geometric Frobenius element at $v$, then the characteristic
polynomial of $r_\lambda(\pi)(\Frob_v)$
is \[X^n-t_v^{(1)}X^{n-1}+\dots+(-1)^jt_v^{(j)}(\mathbf{N}
v)^{j(j-1)/2}X^{n-j}+\dots+(-1)^jt_v^{(n)}(\mathbf{N}
v)^{n(n-1)/2}X^n,\] where the $t_v^{(j)}$ are the eigenvalues of the usual
Hecke operators on $\pi_v^{\GL_2(\cO_{F_v})}$.

\medskip

If $\rho: G \rightarrow \GL(V)$ is any semi-simple two dimensional irreducible representation
which is induced from an index two subgroup $G'$ of $G$, then, by abuse of notation, we call
$r$ \emph{dihedral}. The image of a dihedral representation is a generalized dihedral group; equivalently,
the \emph{projective} image of  $\rho$ in $\PGL(V)$ is a dihedral group.

\subsection{Oddness}We now recall from section 2.1 of~\cite{blggt} the
notion of oddness for essentially self-dual representations of
$G_F$. Let $l>2$ be a prime number, and let $K=\Qlbar$ or $\Flbar$. If
$r:G_F\to\GL_n(K)$ and $\mu:G_F\to K^\times$ are
continuous homomorphisms, then we say that the pair $(r,\mu)$ is
essentially self-dual if for some (so any) infinite place $v$ of $F$
there is an $\epsilon_v\in\{\pm 1\}$ and a  non-degenerate pairing $\langle,\rangle$ on $K^n$ such that
\[\langle x,y\rangle=\epsilon_v\langle y,x\rangle\] and \[\langle
r(\sigma)x, r(c_v\sigma c_v)y\rangle=\mu(\sigma)\langle x,y\rangle\]
for all $x$, $y\in K^n$ and all $\sigma\in G_F$. Equivalently,
$(r,\mu)$ is essentially self-dual if and only if either
$\mu(c_v)=-\epsilon_v$ and $r$ factors through $\GSp_n(K)$ with
multiplier $\mu$, or $\mu(c_v)=\epsilon_v$ and $r$ factors through
$\GO_n(K)$ with multiplier $\mu$.

We say that the pair $(r,\mu)$ is \emph{odd} and essentially self-dual if it is
essentially self-dual, and $\epsilon_v=1$ for all $v|\infty$.

 \begin{lem}\label{lem: odd-dimensional implies odd}
    If $(r,\mu)$ is essentially self-dual and $n$ is odd, then $(r,\mu)$ is odd.
  \end{lem}
  \begin{proof}
    Since $n$ is odd, $r$ factors through $\GO_n(K)$ with multiplier
    $\mu$. Taking determinants, we see that for each $v|\infty$,
    $\mu(c_v)^n=1$, so that $\mu(c_v)=1$, as required.
  \end{proof}
We also have the following trivial lemma.
\begin{lem}
  \label{lem: one-dimensional implies odd esd}If $\chi:G_F\to
  K^\times$ is a character, then $(\chi,\chi^2)$ is odd and
  essentially self-dual.
\end{lem}

We have the following important result of~\cite{belchen}.
\begin{thm}\label{thm: bellaiche chenevier oddness}Let $(\pi,\chi)$ be
  a RAESDC automorphic representation of $\GL_n(\A_F)$, and denote the corresponding compatible
  systems of Galois representations by  $(r_\lambda(\pi),r_\lambda(\chi))$. If for some $\lambda$ we have an
  irreducible subrepresentation $r$ of $r_\lambda(\pi)$ with $r\cong
  r^\vee\otimes\epsilon^{1-n}r_\lambda(\chi)$, then $(r,\epsilon^{1-n}r_\lambda(\chi))$ is essentially
  self-dual and odd. 
\end{thm}
\begin{proof}
  This is Corollary 1.3 of~\cite{belchen}.
\end{proof}
Since  $r_\lambda(\pi)\cong
  r_\lambda(\pi)^\vee\otimes\epsilon^{1-n}r_\lambda(\chi)$, if $r$ is
  an irreducible subrepresentation of $r_\lambda(\pi)$ then there must
  be an irreducible subrepresentation $r'$ of $r_\lambda(\pi)$
  (possibly equal to $r$) with $r'\cong
  r^\vee\otimes\epsilon^{1-n}r_\lambda(\chi)$. In particular, we have:

\begin{cor}
  \label{cor: subreps of large dimension are odd esd}Let $(\pi,\chi)$ be
  a RAESDC automorphic representation of $\GL_n(\A_F)$, and denote the corresponding compatible
  systems of Galois representations by  $(r_\lambda(\pi),r_\lambda(\chi))$. If for some $\lambda$ we have an
  irreducible subrepresentation $r$ of $r_\lambda(\pi)$ with $\dim r>n/2$,
  then $(r,\epsilon^{1-n}r_\lambda(\chi))$ is essentially self-dual and odd.
\end{cor}
\begin{proof}There is an irreducible subrepresentation $r'$ of $r_\lambda(\pi)$ with $r'\cong
  r^\vee\otimes\epsilon^{1-n}r_\lambda(\chi)$; but $\dim r+\dim
  r'>\dim r_\lambda(\pi)$, so we must have $r'=r$. The result then
  follows from Theorem \ref{thm: bellaiche chenevier oddness}.
  \end{proof}
 
Suppose now that  $r:G_F\to\GL_2(\Qlbar)$. Then $r$ factors through
$\GSp_2(\Qlbar)$ with multiplier $\det r$, so the pair $(r,\det r)$ is
essentially self-dual and odd if  $\det r(c_v)=-1$ for all
$v|\infty$. We have the following variant on Theorem 1.2 of
\cite{frank}.

\begin{prop}
  \label{prop: 2d representations are odd}Suppose that $l > 7$ is prime,
  and that $r:G_F\to\GL_2(\Qlbar)$ is a continuous representation. Assume that
  \begin{itemize}
  \item $r$ is unramified outside of finitely many primes.
  \item $\Sym^2\rbar|_{G_{F(\zeta_l)}}$ is irreducible.
  \item $l$ is unramified in $F$.
  \item For each place $v|l$ of $F$ and each $\tau:F_v\into\Qlbar$,
    $\HT_\tau(r|_{G_{F_v}})$ is a set of 2 distinct integers whose
    difference is less than $(l-2)/2$, and $r|_{G_{F_v}}$ is crystalline.
  \end{itemize}
Then the pair $(r,\det r)$ is essentially self-dual and odd.
\end{prop}
\begin{proof}Consider the representation $s=\Sym^2 r$. Since the pair
  $(r,\det r)$ is essentially self-dual, so is the pair $(s,(\det
  r)^2)$. By Lemma \ref{lem: odd-dimensional implies odd}, $(s,(\det
  r)^2)$ is odd. By Corollary 4.5.2 and Lemma 1.4.3(2) of~\cite{blggt}, there
  is a Galois totally real extension $F'/F$ such that
  $(s|_{G_{F'}},(\det r)^2|_{G_{F'}})$ is automorphic. By Proposition
  A of~\cite{tsign}, for any place $v|\infty$ of $F'$ we have \[\tr
  s|_{G_{F'}}(c_v)=\pm 1,\] so that \[\det r|_{G_{F'}}(c_v)= -1,\]and
  $(r,\det r)$ is odd, as required.
 \end{proof}

\subsection{Residually dihedral compatible systems}We now
show that residually dihedral compatible systems are themselves
dihedral up to a set of places of density zero.
\begin{lem}
  \label{lem: Fontaine-Laffaille dihedral residual representations are
    unramified} Suppose that $l>2$ is unramified in $F$, and that
  $s:G_F\to \GL_2(\Qlbar)$ is a continuous irreducible representation,
  such that if $v|l$ is a place of $F$, then
  \begin{itemize}
  \item $s|_{G_{F_v}}$ is crystalline, and
  \item for all embeddings $F\into\Qlbar$, $\HT_\tau(s|_{G_{F_v}})$
    consists of two distinct integers with difference less than $(l-2)/2$.
  \end{itemize}
Assume that $\sbar$ is dihedral, so that $\sbar$ is induced from a
character of a quadratic extension $K/F$. Then $l$ is unramified in $K$.
\end{lem}
\begin{proof}
  Let $v|l$ be a place of $F$. Assume for the sake of a contradiction
  that $v$ is ramified in $K$. If $\sbar|_{G_{F_v}}$ is irreducible,
  then $\sbar|_{G_{F_v}}$ is induced from a character $\chi$ of $G_L$, where $L$
  is the quadratic unramified extension of $F_v$. But
  then \[\sbar|_{G_{K_v}}\cong(\Ind_{G_{L}}^{G_{F_v}}\chi)|_{G_{K_v}}\cong
  \Ind_{G_{LK_v}}^{G_{K_v}}\chi|_{G_{LK_v}}\] is irreducible, a
  contradiction.

If on the other hand $\sbar|_{G_{F_v}}$ is reducible, then since it is
isomorphic to the induction from $K_v$ of some character, it must be
of the form $\psi_1\oplus\psi_2$ with $\psi_1\psi_2^{-1}$
quadratic. Let $k$ be the residue field of $F_v$, and for each
$\sigma:k\into\Flbar$ let $\omega_\sigma$ be the corresponding
fundamental character of $G_{F_v}$ of niveau $1$. By
Fontaine-Laffaille
theory, \[\psi_1\psi_2^{-1}|_{I_{F_v}}=\prod_{\sigma:k\into\Flbar}\omega_\sigma^{a_\sigma}\]where
$a_\sigma$ is the (positive or negative) difference between the
elements of  $\HT_\tau(s|_{G_{F_v}})$, where $\tau:F_v\into\Qlbar$ is
the unique lift of $\sigma$. By assumption, we have $2a_\sigma\in
[2-l,l-2]$, so $(\psi_1\psi_2^{-1})^2|_{I_{F_v}}\ne 1$, a
contradiction. So $l$ is unramified in $K$, as claimed.
\end{proof}

\begin{prop}
  \label{prop: residually dihedral compatible system implies globally
    dihedral and thus odd}Suppose that $\cR$ is a regular, weakly
  compatible system of $l$-adic representations of $G_F$. Then there
  is a set of rational primes $\cL$ of density one such that if
  $\lambda$ lies over a place of $\cL$ and $s$ is a two-dimensional
  irreducible summand of $r_\lambda$ such that $\sbar$ is dihedral,
  then $s$ is also dihedral, and the pair $(s,\det s)$ is essentially
  self-dual and odd.
\end{prop}
\begin{proof}
  Note firstly that if $s$ is dihedral, then it is induced from an
  algebraic character of a quadratic extension of $F$. If this
  quadratic extension is not totally imaginary, then this character would be
  a finite order character times a power of the cyclotomic character,
  contradicting the regularity of $s$. So the extension must be
  totally imaginary, in which case $\det s(c_v)=-1$ for all places
  $v|\infty$ of $F$, and the pair $(s,\det s)$ is essentially
  self-dual and odd.

Let $S$ be the finite set of primes at which $\cR$ ramifies, and let
$F'$ be the maximal abelian extension of $F$ of exponent 2 which is
unramified outside $S$ (the extension $F'/F$ is finite). By Lemma
\ref{lem: Fontaine-Laffaille dihedral residual representations are
  unramified}, for all but finitely many $\lambda$, if $s$ is as in
the statement of the proposition then $\sbar|_{G_{F'}}$ is
reducible. Applying Proposition~5.2.2 of~\cite{blggt} to the regular
weakly compatible system $\cR|_{G_{F'}}$, we see that there is a set
$\cL$ of rational primes of density one such that if $\lambda$ lies
over an element of $\cL$ and $s$ is as in the statement of the
proposition, then $s|_{G_{F'}}$ is reducible, so that $s$ is dihedral,
as required.
\end{proof}

\section{Irreducibility for a density one set of primes}

\subsection{}In this section, we establish the irreducibility
of $r_{\lambda}(\pi)$ for a density one set of primes $\lambda$.

\begin{prop}\label{prop: potentially automorphic implies irreducible}
 Let $F$ be a totally real field. Suppose that $\pi$ is a RAESDC automorphic representation of
  $\GL_n(\A_F)$, and that for some $\lambda$ we have a
  decomposition \[r_\lambda(\pi)=r_\lambda(\pi)_1\oplus\dots\oplus
  r_\lambda(\pi)_j,\]where each $r_\lambda(\pi)_i$ is
  irreducible. Suppose also that there is a totally real Galois
  extension $F'/F$ such that each $r_\lambda(\pi)_i|_{G_{F'}}$ is
  irreducible and automorphic. Then $j=1$, so $r_\lambda(\pi)$ is irreducible.
\end{prop}
\begin{proof}
  This may be proved by an identical argument to the proof of Theorem
  5.4.2 of~\cite{blggt}.
\end{proof}

\begin{thm}
  \label{thm: irreducibility for density one}Let $F$ be a totally real
  field. Suppose that $(\pi,\chi)$ is a
  RAESDC automorphic representation of $\GL_n(\A_F)$ with $n\le
  5$. Then there is a density one set of rational primes $\cL$ such that
  if $\lambda$ lies over a prime in $\cL$, then $r_\lambda(\pi)$ is irreducible.
\end{thm}
\begin{proof}Write \[r_\lambda(\pi)=r_\lambda(\pi)_1\oplus\dots\oplus
  r_\lambda(\pi)_{j_\lambda},\]with each $r_\lambda(\pi)_i$
  irreducible. By Proposition~5.2.2 of~\cite{blggt} there is a density one set of
  rational primes $\cL$ such that if $\lambda$ lies over a prime of
  $\cL$, then each $\rbar_\lambda(\pi)_i|_{G_{F(\zeta_l)}}$ is
  irreducible. We may assume that every prime in $\cL$ is at least 13.

If $\dim r_\lambda(\pi)_i\ge 3$, then by Corollary \ref{cor: subreps
  of large dimension are odd esd} and the hypothesis that $n\le 5$ we
see that the pair $(r_\lambda(\pi)_i,\epsilon^{1-n}r_\lambda(\chi))$
is essentially self-dual and odd. If $\dim r_\lambda(\pi)_i=1$, then
by Lemma \ref{lem: one-dimensional implies odd esd}, the pair
$(r_\lambda(\pi)_i,r_\lambda(\pi)_i^2)$ is essentially self-dual and
odd. 

Suppose now that $\dim r_\lambda(\pi)_i=2$. By removing finitely many
primes from $\cL$, we see from Proposition \ref{prop: 2d
  representations are odd} that we may assume that if $\lambda$ lies
over an element of $\cL$, and $\Sym^2
\rbar_\lambda(\pi)_i|_{G_{F(\zeta_l)}}$ is irreducible, then the pair
$(r_\lambda(\pi)_i,\det r_\lambda(\pi)_i)$ is essentially self-dual
and odd. If $\lambda$ lies over an element of $\cL$ and
$\Sym^2\rbar_\lambda(\pi)_i|_{G_{F(\zeta_l)}}$ is reducible, then
since $\rbar_\lambda(\pi)_i|_{G_{F(\zeta_l)}}$ is irreducible, it
follows from Lemmas 4.2.1 and 4.3.1 of~\cite{blggord} that
$\rbar_\lambda(\pi)_i$ has dihedral image. By Proposition \ref{prop:
  residually dihedral compatible system implies globally dihedral and
  thus odd}, after possibly replacing $\cL$ with a subset of
density one, the pair $(r_\lambda(\pi)_i,\det r_\lambda(\pi)_i)$ is
essentially self-dual and odd.

Thus if $\lambda$ divides a prime in $\cL$, for each $i$ there is a
character $\chi_{\lambda,i}$ such that the pair
$(r_\lambda(\pi)_i,\chi_{\lambda,i})$ is essentially self-dual and
odd. After possibly removing a finite set of primes from $\cL$, we may
assume that every element of $\cL$ is unramified in $F$, and that each
$r_\lambda(\pi)_i$ is crystalline with Hodge-Tate weights in the
Fontaine-Laffaille range. Fix some $l\in\cL$ and some $\lambda|l$. Let
$K$ be an imaginary quadratic extension of $F$ in which each place of
$F$ above $l$ splits completely. By Theorem 4.5.1 of~\cite{blggt},
there is a finite Galois CM extension $K'$ of $K$ such that each
$r_\lambda(\pi)_i|_{G_{K'}}$ is irreducible and automorphic. 

Let $F'$ be the maximal totally real subfield of $K'$. By Lemma 1.5 of
\cite{blght} each $r_\lambda(\pi)_i|_{G_{F'}}$ is irreducible and
automorphic. The result now follows from Proposition \ref{prop:
  potentially automorphic implies irreducible}.
 \end{proof}

\section{Irreducibility for all primes}\subsection{}In this section we
prove that the representations $r_\lambda(\pi)$ are irreducible for
all $\lambda$.

\begin{thm}\label{thm: all irred char 0}Let $F$ be a totally real
  field. Suppose that $(\pi,\chi)$ is a
  RAESDC automorphic representation of $\GL_n(\A_F)$ with $n\le
  5$. Then all  of the representations
  $r_{\lambda}(\pi)$ are irreducible. \end{thm}
By Theorem~\ref{thm: irreducibility for density one},
  $r_{\lambda}(\pi)$
is irreducible for a set of $\lambda$ of density one. By 
 Proposition~5.2.2 of~\cite{blggt}, this implies that 
 $\rbar_{\lambda}(\pi)$ is irreducible for a set of $\lambda$ of density one.

\medskip

\begin{defn}
  Let $F$ be a number field. We say that a representation
  $r:G_F\to\GL_n(\Qlbar)$ is \emph{strongly irreducible} if for all
  finite extensions $E/F$, $r|_{G_E}$ is irreducible.
\end{defn}

We would like to understand when
the Galois representations $r_{\lambda}(\pi)$ which are irreducible
can fail to be strongly irreducible.
We begin with an easy group theory lemma.

\begin{lemma} \label{lem:inductions}
Suppose that $G$ acts irreducibly on a finite dimensional vector space $V$ of dimension~$n$.
Let $G'$ be a normal finite index subgroup of $G$, and suppose that 
$V  |_{G'} \simeq \bigoplus W_k$ decomposes non-trivially as a $G'$ representation into $m$ distinct irreducible representations.
Then $m|n$ and there exists a  proper subgroup $H \supseteq G'$ of $G$ of index $m$ and an irreducible representation
$W$ of $H$ such that $V \simeq \Ind^{G}_{H} W$.
\end{lemma}

\begin{proof} Since the representations $W_k$ are distinct, and since $G'$ is normal, the group $G$
acts transitively on the set of representations $W_k$. In particular, all the $W_k$ have the same dimension.
 Let $W$ be one of these representations, and let
$H$ denote the stabilizer of $W$. By the orbit--stabilizer theorem, the index of $H$ in $G$ is
$m$. The representation $W$ extends to a representation of $H$. Since $\Hom_{H}(V,W)$ is non-trivial,
by Frobenius reciprocity $\Hom_{G}(V,\Ind^{G}_{H} W)$ is also non-trivial. Yet
$\Ind^{G}_{H}(W)$ has dimension $[G:H] \dim(W) = \dim(V)$ and $V$ is irreducible, and thus the homomorphism
$V  \rightarrow \Ind^{G}_{H}(W)$ must be both an injection and a surjection, and hence an isomorphism. \end{proof}

Using this lemma, we shall see that the density one set of irreducible Galois representations
$r_{\lambda}(\pi)$  remain irreducible upon restriction to any fixed finite extension,
\emph{except} in situations in which we can prove Theorem~\ref{thm: all irred char 0}
directly.

\begin{corollary} \label{cor:induct}Let $F$ be a totally real field. Let $(\pi,\chi)$ be a RAESDC
  automorphic representation of $\GL_n(\A_F)$ with $n \le 5$.  Let
  $\lambda$ be a prime such that $r_{\lambda}(\pi)$ is irreducible.
Then either:
\begin{enumerate} 
\item $r_\lambda(\pi)$ is strongly irreducible, or
\item $(\pi,\chi)$ is an automorphic induction,  $r_{\lambda}(\pi)$ is irreducible for all $\lambda$, and
$\rbar_{\lambda}(\pi)$ is irreducible for all but finitely many $\lambda$.
\end{enumerate}
\end{corollary}

\begin{proof}
  We claim that for any finite extension $E/F$, either
  $r_{\lambda}(\pi) |_{G_{E}}$ is irreducible or it decomposes into a
  sum of \emph{distinct} irreducible representations. This follows
  immediately from the fact that $r_{\lambda}(\pi)|_{G_{E}}$ has
  distinct Hodge--Tate weights at any prime $w|l$.  (Note that
  $r_{\lambda}(\pi) |_{G_{E}}$ is necessarily semisimple; if $V$
  denotes any irreducible subrepresentation, then the various
  conjugates of $V$ are stable under the conjugates of $G_E$, and we
  see that $r_{\lambda}(\pi) |_{G_{E}}$ becomes completely
  decomposable under restriction to a finite index subgroup, so must
  already have been semisimple.) Suppose that $r_\lambda(\pi)|_{G_E}$
  is reducible for some finite extension $E/F$. Replacing $E$ by its
  normal closure over $F$, we may assume that the extension $E/F$ is
  Galois, and hence by Lemma~\ref{lem:inductions}, we see that
  $r_{\lambda}(\pi)$ is the induction of an irreducible representation
  from some finite extension of degree dividing $n$. If $n =2$, $3$ or
  $5$, the only possibility is that $r_{\lambda}(\pi)$ is the
  induction of a character from some degree $n$ extension $H$ of $F$.
  This character is de Rham, and is thus the Galois representation
  associated to an algebraic Hecke character. If $n=3$ or $5$, then
  $[H:F]$ is odd, and thus $H$ does not contain a CM field. It follows
  that the corresponding Galois representation is a finite order
  character times some power of the cyclotomic character. This
  contradicts the regularity of $r_{\lambda}(\pi)$. If $n=2$, then $H$
  must be a CM field, and $r_\lambda(\pi)$ is the induction of an
  algebraic Hecke character. The claims regarding the irreducibility
  of $\rbar_{\lambda}(\pi)$ are elementary to verify in this case.
  
  If $n = 4$, then either
  $r_{\lambda}(\pi)$ is the induction of a character from some degree
  $4$ extension $L/F$, or the induction of a two dimensional
  representation of some quadratic extension $K/F$.  In the first
  case, since $r_{\lambda}(\pi)$ is regular, $L$ contains a CM field.
  It follows that $L$ must contain a subfield $K$ of index two, and
  thus in both cases $r_{\lambda}(\pi)$ is the induction of some two
  dimensional representation of some quadratic extension $K/F$.  It
  follows that there is an isomorphism $r_{\lambda}(\pi) \simeq
  r_{\lambda}(\pi) \otimes \eta$, where $\eta$ is the quadratic
  character of $K/F$. By multiplicity one for $\GL_4(\A_{F})$
  (\cite{MR623137}) and by Theorem~4.2 (p.202) of~\cite{MR1007299}, we
  deduce that $(\pi,\chi)$ is an automorphic induction from some
  quadratic field $K/F$.
 
  It suffices to prove that when $n = 4$ and $(\pi,\chi)$ is an
  induction of some cuspidal automorphic representation $\varpi$ of
  $\GL_2(\A_K)$ from some quadratic field $K/F$, then
  $r_{\lambda}(\pi)$ is irreducible for \emph{all} $\lambda$, and
  $\rbar_{\lambda}(\pi)$ is irreducible for all but finitely many
  $\lambda$.  If $K/F$ is totally real, then $\varpi$ corresponds to a
  Hilbert modular form with corresponding Galois representations
  $s_{\lambda}(\varpi)$, and there are isomorphisms $r_{\lambda}(\pi)
  \simeq \Ind^{G_F}_{G_K} s_{\lambda}(\varpi)$.  The representation
  $s_{\lambda}(\varpi)$ is irreducible for all $\lambda$, and
  $\sbar_{\lambda}(\varpi)$ is irreducible for all but finitely many
  $\lambda$ (\cite[Proposition 0.1]{MR2172950}).  If $r_{\lambda}(\pi)$ is reducible, then
  $s_{\lambda}(\varpi) \simeq s^c_{\lambda}(\varpi)$ where $c$ is the
  non-trivial element of $\Gal(K/F)$. Similarly, if
  $\rbar_{\lambda}(\pi)|_{G_F}$ is reducible for infinitely many
  primes $\lambda$, then $\sbar_{\lambda}(\varpi) \simeq
  \sbar^c_{\lambda}(\varpi)$ for infinitely many $\lambda$. In either
  case, by multiplicity one, we deduce that $\varpi \simeq \varpi^c$,
  and by Theorem~4.2 (p.202) of~\cite{MR1007299}, we deduce that
  $\varpi$ itself arises from base change from $\GL_2(\A_F)$.  If this
  is so, however, then $\pi$ is not cuspidal, contrary to
  assumption. Suppose instead that $K/F$ is not totally real. By (3.6.1)
  of~\cite{MR892187}), the infinitesimal character of $\varpi$ at any
  pair of complex conjugate infinite places must be equal,
  contradicting the regularity of $\pi$. \end{proof}

In the sequel, it will be useful to collate some information
about irreducible representations of  semi-simple  Lie algebras of  small dimension.

\begin{prop} \label{prop: list of algebras}
Let $k$ be an algebraically closed field of characteristic zero.
Let $G$ be the $k$-points of a reductive algebraic group acting faithfully
and irreducibly on a vector space over $k$ of dimension $n$. 
Let $G^0$ denote the connected component of $G$, 
let $\gog$ be the Lie algebra of $G^0$, and write $\gog = \gz \oplus
\gh$, with $\gz$ is abelian and $\gh$ semisimple.
Suppose that $G^0$ is not abelian. Then, for $n \le 6$, $\gh$ is one
of the following algebras, where the columns of the table  correspond to whether $G$
preserves a generalized orthogonal pairing, a generalized
symplectic pairing, or does \emph{not} preserve any such pairing respectively.
\emph{
\begin{center}
\begin{tabular}{|c|c|c|c|}
\hline
  & $\GO$ & $\GSp$ & $\GL$ \\
\hline
 $2$ & *  & $\slf_2$  & *  \\
$3$ & $\slf_2$ & * & $\slf_3$ \\
$4$ & $\sof_4 = \slf_2 \times \slf_2$ & $\slf_2$, $\slf_2 \times \slf_2$,
$\spf_4$ & $\slf_4$ \\
$5$ & $\slf_2$, $\sof_5 = \spf_4$ & * & $\slf_5$ \\
$6$ & $\sof_6$ & $\slf_2$, $\slf_2 \times \slf_2$, $\spf_6$ & $\slf_2 \times
\slf_3$, $\slf_2 \times \slf_2 \times \slf_2$, $\slf_3 \times \slf_3$, $\slf_4$, $\slf_6$ \\
\hline
\end{tabular}
\end{center}
}
\end{prop}

\begin{proof} This can be checked directly by hand.
In dimension $n$, the representation of $\slf_2$ is the $(n-1)$st symmetric
power of the tautological representation, which is orthogonal
if $n$ is odd and symplectic if $n$ is even. 
The four dimensional symplectic representation of $\slf_2 \times \slf_2$ is reducible,
but the image  of $\GL(2) \times \GL(2)$ in $\GL(4)$ is normalized
by a group $G$ which contains it with index two and does act irreducibly.
The same is true of $\slf_3 \times \slf_3$ (index two) and
$\slf_2 \times \slf_2 \times \slf_2$ (index six) in dimension six.
The algebra $\slf_2 \times \slf_2$ has a six dimensional
representation which is the tensor product of the standard
representation of the first factor and the symmetric square of the second.
Finally, $\slf_4$ has a six dimensional representation which is
$\wedge^2$ of the tautological representation.
\end{proof}

The key idea of our argument is that,
 with certain caveats, we can detect reducibility
on the level of compatible systems. Suppose that $\cR$ and $\cS$ are two weakly compatible systems of 
Galois representations.
\begin{df}  Say that $\cS$ \emph{weakly divides} $\cR$ if
the characteristic polynomials of $\cS$ 
divide the characteristic polynomials of $\cR$.
\end{df}

It is not true that if $\cS$ weakly divides $\cR$ then there is a
corresponding splitting of Galois representations. A good example to
keep in mind is as follows. Suppose that $\pi$ is a RAESDC automorphic
representation for $\GL_2(\A_F)$ with central character $\psi$ which
does not have CM.  Then if $\cR$ and $\cS$ are the compatible systems
associated to $\Sym^{n-2}(\pi) \otimes \psi$ and $\Sym^{n}(\pi)$
respectively then $\cS$ weakly divides $\cR$, but both compatible
systems are irreducible.  Nevertheless, we will be able to detect non-trivial
information from weak divisibility. The key result is the following.

\begin{thm} \label{thm: punt}
Let $V$ be a vector space of dimension $n \le 5$ over an
algebraically closed field of characteristic zero.
Let $H \subseteq \GL(V)$ be a  Zariski closed subgroup, and assume that the
connected component of the identity $H^0$ acts irreducibly on $V$.  Suppose that every
$h \in H$ has a fixed vector in $V$. Then  
$H = \PSL(2)$, $n$ is odd, and $V \simeq \Sym^{n-1}(W)$ where
$W$ is the standard representation of $\SL(2)$.
For a generic semi-simple element $h \in H$, one has $\dim(V|h = 1) = 1$.
\end{thm}

\begin{proof} Tautologically, $H$ admits a faithful irreducible representation
into $\GL(V)$, and thus $H$ is reductive. Let $Z$ be the center of $H$.
By Schur's lemma, $Z$ acts on $V$ via scalars. Yet (by assumption) any $z \in Z \subset H$
has a fixed vector, and thus $Z$ is trivial. In particular, $H$ is semi-simple.
It suffices to assume that 
every $t \in T \subset H$ has a fixed vector for every torus $T$ on $H$.
Since $H^0$ is connected, 
we may check this condition for $H^0$
on the level of Lie algebras. The only
(semi-simple) Lie algebras $\gh$ which admit faithful irreducible representations of
dimension at most five are $\slf_2$,  $\slf_3$, $\sof_4 = \slf_2 \times \slf_2$,  $\spf_4 = \sof_5$, 
$\slf_4$, and $\slf_5$. It is easy to check that the only possibility
 is that $\gh = \slf_2$. Since $H^0$ is connected with trivial center it must be the adjoint form of $\SL(2)$,
 which is $\PSL(2)$.  The irreducible representations of $\PSL(2)$ are given by the even symmetric
 powers of the standard representation of $\SL(2)$. 
 The group $H^0 = \PSL(2)$ has a trivial outer automorphism group. Hence the action of conjugation by
  $h \in H$ on $H^0$ is given by conjugation by a element $\gamma$ of $\PSL(2)$.  It follows that $\gamma^{-1} h$
  acts trivially on $\PSL(2)$, and thus, by Schur's lemma, $\gamma^{-1} h$ is a scalar, and so lies in the center $Z$ of $H$.
  Yet we have seen that $Z$ is trivial, and so it follows
  that $H = H^0$. A generic semi-simple element in $\PSL(2)$ is conjugate to an element of the form
 $h=\displaystyle{\left(\begin{matrix} z &  \\ & z^{-1} \end{matrix} \right)}$, for which $\dim(V|h=  1) = 1$
 unless $z$ is a root of unity of sufficiently small order.
 \end{proof}
 
 Using this, we prove the following.

\begin{lemma}\label{lem: 2 divides 5}Let $F$ be a totally real field. Let $(\pi,\chi)$ be a RAESDC automorphic
representation of $\GL_n(\A_F)$ with $n \le 5$, and let
$\cR$ be the corresponding weakly compatible system.
Suppose that either:
\begin{enumerate}
\item $\cR$ is weakly divisible by a compatible system  of algebraic Hecke characters.
\item For some finite extension $E/F$, $\cR|_{G_E}$ is weakly divisibly by a direct sum of compatible
systems of two algebraic Hecke characters over $E$.
\end{enumerate}
Then either $(\pi,\chi)$ is an automorphic induction, or $n$ is odd,
and there exists a finite Galois extension $F'/F$ and a compatible
system of two dimensional irreducible Galois representations $\cS$ of
$G_{F'}$ such
that $\cR|_{G_{F'}}$ is 
a twist of $\Sym^{n-1}(\cS)$. In either case, $r_{\lambda}(\pi)$ is irreducible
for all $\lambda$, and $\rbar_{\lambda}(\pi)$ is irreducible for all but finitely many $\lambda$.

\end{lemma}

\begin{proof} 
  We may assume that $(\pi,\chi)$ is not an automorphic induction. By
  Corollary~\ref{cor:induct}, we may find a sufficiently large
  $\lambda$ such that $r_{\lambda}(\pi)$ is strongly irreducible.  We
  may also assume that $\rbar_{\lambda}(\pi)$ is irreducible.  Let us
  assume that we are in case $(1)$.  After twisting, we may assuming
  that the compatible system $\cR$ is divisible by the trivial
  character.  Let $H$ denote the Zariski closure of the image of
  $r_{\lambda}(\pi)$, and let $H^0$ denote the connected component of
  $H$. Since Frobenius elements are Zariski dense, we deduce that
  every $h \in H$ has a fixed vector. We deduce from Theorem~\ref{thm:
    punt} that $n$ is odd, and that the image of $r_{\lambda}(\pi)$
  lands in the image of $\PSL_2(\Qbar_l)$ in $\GO_n(\Qbar_l)$ under
  the $(n-1)$-st symmetric power map.  In particular, the image of
  $r_{\lambda}(\pi)$ lands in $\mathrm{SO}_n(\Qbar_l)$ and
  $\det(r_{\lambda}(\pi)) = 1$.  The obstruction to lifting a
  projective homomorphism from $\PSL_2(\Qbar_l)$ to $\GL_2(\Qbar_l)$
  lies in $H^2(G_F,\Qlbar^\times)$, which vanishes by a result of Tate
  (for example, see Theorem 5.4 of~\cite{MR2461903}).  Hence there
  exists a Galois representation $s_{\lambda}: G_{F} \rightarrow
  \GL_2(\Qbar_l)$ and an isomorphism
$$r_{\lambda}(\pi) = \Sym^{n-1} s_{\lambda}  \otimes \det(s_{\lambda})^{\frac{1-n}{2}}.$$ 
We will show that $s_{\lambda}$ is potentially modular, and use known
irreducibility results for the Galois representations associated to
Hilbert modular forms to conclude.

We first show that $\ad^0(s_{\lambda}) = \Sym^2(s_{\lambda})
\det(s_{\lambda})^{-1}$ is de Rham. We see this from the following
plethysm for the standard representation $W$ of $\GL(2)$:
$$(\Sym^{n-1} W)^{\otimes 2} \otimes \det(W)^{-(n-1)} = \bigoplus_{i=0}^{n-1} \Sym^{2i} W \otimes
\det(W)^{-i}.$$ In particular, $\ad^0(s_{\lambda})$ is a constituent
of $(r_{\lambda}(\pi))^{\otimes 2}$, and hence is crystalline with
Hodge--Tate weights in the Fontaine--Laffaille range for sufficiently
large $\lambda$. We claim that in fact $\ad^0(s_\lambda)$ is regular;
in the cases $n=1$, $3$ this is immediate from the regularity of
$r_\lambda(\pi)$, and in the case $n=5$ it follows from the regularity
of $r_\lambda(\pi)$ and the
relation \[\Sym^2(\ad^0(s_\lambda))=r_\lambda(\pi)\oplus 1.\] Since
$r_{\lambda}(\pi)$ is odd, it follows immediately that $s_{\lambda}$
and hence $\ad^0(s_{\lambda})$ is also odd. Since it is 2-dimensional,
$s_\lambda$ is automatically essentially self-dual, so that
$\ad^0(s_\lambda)$ is also essentially self-dual.  If
$\sbar_{\lambda}$ was dihedral or reducible, then
$\rbar_{\lambda}(\pi)$ would be reducible, contrary to
assumption. Hence if $\lambda$ is sufficiently large,
$\ad^0(\sbar_{\lambda})|_{G_{F(\zeta_l)}}$ is irreducible, and so we
deduce from Corollary 4.5.2 and Lemma 1.4.3(2) of \cite{blggt} that
there is a finite Galois extension of totally real fields $F'/F$ such
that $\ad^0(s_{\lambda})|_{G_{F'}}$ is automorphic. It follows that up
to twist $s_{\lambda}|_{G_{F'}}$ itself is automorphic (using the
characterization of the image of the symmetric square in Theorem~A and
Corollary~B of~\cite{ramakrishnan2009selfdual}), say
$s_\lambda|_{G_{F'}}\cong r_\lambda(\pi')\otimes\psi$ for some
$\psi$. Since $r_\lambda(\pi)$ is strongly irreducible,
$r_\lambda(\pi)|_{G_{F'}}\cong\Sym^{n-1}r_\lambda(\pi')\otimes\det(r_{\lambda}(\pi'))^{\frac{1-n}{2}}$
is irreducible, so $\pi'$ cannot be of CM type. Thus, for all
$\lambda'$, we have
$r_{\lambda'}|_{G_{F'}}\cong\Sym^{n-1}r_{\lambda'}(\pi')\otimes\det(r_{\lambda}(\pi'))^{\frac{1-n}{2}}$
is irreducible, and $\rbar_{\lambda'}|_{G_{F'}}$ is irreducible for
all but finitely many $\lambda'$ (since for all but finitely many
$\lambda'$, the image of $\rbar_{\lambda'}|_{G_{F'}}$ will contain
$\SL_2(\F)$ for some finite field $\F$ by \cite[Proposition
0.1]{MR2172950}).

Suppose instead that we are in case $(2)$. After twisting, we may assume that the
compatible system $\cR|_{G_E}$ is divisible by the trivial character.
If the second character is also trivial (after this twist), we obtain
a contradiction with Theorem~\ref{thm: punt}, since the generic
multiplicity of the $h = 1$ eigenspace is $1$.  Hence the characters
are different. It follows as in the first paragraph of this proof that both the representations
$r_{\lambda}(\pi)$ and $r_{\lambda}(\pi) \otimes \chi$ have trivial
determinant for some $\chi \ne 1$, which implies that $\chi$ has
finite order. Replacing $E$ with the fixed field of the kernel of $\chi$, we reduce
to the case that both characters are trivial, which is a contradiction.
\end{proof}

\begin{cor}
  \label{cor: divisibility by a single character is impossible}Let $F$
  be a totally real field. Let $(\pi,\chi)$ be a RAESDC automorphic
representation of $\GL_n(\A_F)$ with $n \le 5$. Suppose that some
$r_\lambda(\pi)$ is reducible, say $r_\lambda(\pi)=s_\lambda\oplus
t_\lambda$. Then $\min(\dim s_\lambda,\dim t_\lambda)\ge 2$, and
neither of $s_\lambda$, $t_\lambda$ can be dihedral.\end{cor}
\begin{proof}Suppose that without loss of generality $\dim
  s_\lambda=1$. Since $r_\lambda(\pi)$ is de Rham, so is $s_\lambda$,
  so by e.g. Lemma 4.1.3 of \cite{cht} there is an algebraic character $\chi$ of $\A_F^\times/F^\times$
  such that $s_\lambda=r_\lambda(\chi)$. Then the weakly compatible system
  $\{r_\lambda(\chi)\}$ weakly divides $\{r_\lambda(\pi)\}$, so by
  Lemma \ref{lem: 2 divides 5}, we see that $r_\lambda(\pi)$ is irreducible for all
  $\lambda$, a contradiction.

Suppose now that $s_\lambda$ is dihedral. Then there is a quadratic
extension $E/F$ such that $s_\lambda|_{G_E}$ is reducible, so is a sum
of two de Rham characters. Arguing in the same way, we again obtain a
contradiction from Lemma \ref{lem: 2 divides 5}.
  \end{proof}
  We now prove Theorem \ref{thm: all irred char 0}. We proceed by
  contradiction, assuming that for some $\lambda$ we have
  $r_\lambda(\pi)=s_\lambda\oplus t_\lambda$. By Corollary \ref{cor:
    divisibility by a single character is impossible}, it suffices to
  consider the cases that both $s_\lambda$ and $t_\lambda$ are
  irreducible, that $n=4$ or $5$, and that $\dim s_\lambda=2$.

\medskip
\subsection{The case $n=5$} Since $5$ is odd, we see from Lemma
\ref{lem: odd-dimensional implies odd} that $r_\lambda(\pi)$ factors
through $\GO_5$ with even multiplier. Since $2\ne 3$, we see that $s_\lambda$
factors through $\GO_2$ with even multiplier, so $s_\lambda$ is
dihedral. This contradicts Corollary \ref{cor: divisibility by a single character is impossible}.  \medskip

\subsection{The $n=4$ symplectic case}\label{subsec:symplectic char 0}Suppose that $r_\lambda(\pi)$
is symplectic with odd multiplier. If $V$ is a (generalized) symplectic representation of a group $G$ with multiplier $\chi$, then
there is a surjection $\wedge^2 V \rightarrow \chi$. Hence the virtual representation
$\wedge^2(V) - \chi$ is an actual representation of $G$.
In particular, if $\cR$ is the compatible system of Galois representations associated to $(\pi,\chi)$,
then $\cA:= \wedge^2(\cR) - \chi$ is a compatible system of Galois representations such that
$a_{\lambda}:G_F \rightarrow \GL_5(\Qbar_l)$ has image in $\GO_5(\Qbar_l)$. Moreover, this
compatible system is odd (automatic since $5$ is odd) and regular.

Since $r_\lambda(\pi)=s_\lambda\oplus t_\lambda$, we have
$$a_{\lambda} \oplus \chi_{\lambda} = s_{\lambda} \otimes t_{\lambda} \oplus 
\det(s_{\lambda}) \oplus \det(t_{\lambda}).$$ In particular, as there
are two characters on the right hand side, the representation
$a_\lambda$ contains a character, and we deduce that the compatible
system $\cA$ is weakly divisible by the compatible system of a character.

For a density one set of primes $\lambda'$, the representation
$r_{\lambda'}(\pi)$ is irreducible.  For such a prime $\lambda'$, let
$G$ denote the Zariski closure of the image, and $G^0$ the connected
component of the identity. Let $Z$ denote the center of $H$. If $G^0$
acts reducibly, then $r_{\lambda'}$ is potentially reducible and hence
$(\pi,\chi)$ is an automorphic induction by
Corollary~\ref{cor:induct}, and we would be done. Hence $G^0$ acts
irreducibly.
Let $\gog$ be the Lie algebra of $G^0$,
and let $\gog = \gh \oplus \gz$ for $\gh$ semi-simple and $\gz$ central. We deduce
that $\gh$ acts irreducibly,
and hence $\gh = \spf_4$ or $\slf_2 \subset \spf_4$ acting through the third symmetric power representation.
In either case, the corresponding representation in dimension $5$ of $\sof_5 = \spf_4$ is irreducible, and
thus $a_{\lambda'}$ is also irreducible, even after restricting to any finite extension $E/F$.
Arguing as in the proof of Lemma~\ref{lem: 2 divides 5}, we deduce
that there is a finite Galois extension $F'/F$ of totally real fields
and a RAESDC automorphic representation $\pi'$ of $\GL_2(\A_{F'})$
such that, for all $\lambda''$, we have
$a_{\lambda''}|_{G_{F'}}\cong\Sym^4r_{\lambda''}(\pi')\otimes\det(r_{\lambda''}(\pi'))^{-2}$.
Since $a_{\lambda'}$ is irreducible, $\pi'$ cannot be of CM type, so
that in fact $a_{\lambda''}$ is irreducible for all $\lambda''$. This
contradicts the reducibility of $a_\lambda$.\medskip

\subsection{The $n=4$ orthogonal case}Suppose finally that
$r_\lambda(\pi)$ is orthogonal with even multiplier, and that
$r_\lambda(\pi)=s_\lambda\oplus t_\lambda$. By Corollary \ref{cor:
  divisibility by a single character is impossible} both $s_\lambda$
and $t_\lambda$ are non-dihedral two-dimensional representations.

Since $r_\lambda(\pi)$ factors through $\GO(4)$, it must either be the
case that each of $s_\lambda$ and $t_\lambda$ factors through $GO(2)$,
or that the orthogonal pairing identifies $s_\lambda$ with
$t_\lambda^\vee\otimes\epsilon^{-3}r_\lambda(\chi)$. In the former
case both $s_\lambda$ and $t_\lambda$ would be dihedral, a
contradiction, so we must be in the latter case. Since we also have
$t_\lambda^\vee\cong t_\lambda\otimes\det(t_\lambda)^{-1}$, we may
write \[r_\lambda(\pi)\cong t_\lambda\oplus
t_\lambda\otimes\psi,\]where
$\psi=\epsilon^{-3}r_\lambda(\chi)\det(t_\lambda)^{-1}$.  Since
$r_\lambda(\pi)$ is de Rham, we see that $t_\lambda$ and $\psi$ are de
Rham. Thus $\psi$ is pure of some weight. However, the representation
$r_\lambda(\pi)$ is pure, so if $v\nmid l$ is a finite place of $F$
with $\pi_v$ unramified, then all of the eigenvalues of
$r_\lambda(\pi)(\Frob_v)$ are Weil numbers of the same weight. This
implies that $\psi$ must be pure of weight $0$; but this contradicts
the regularity of $r_\lambda(\pi)$.

\section{Representations with small image}\subsection{}With an eye to
proving that $\rbar_\lambda(\pi)$ is irreducible for all but finitely
many $\lambda$, in this section, we prove a variety of results on
residually reducible or dihedral 2-dimensional representations, using
class field theory and the main conjecture of Iwasawa theory.

Fix number fields $F$, $M$ such that $|\Hom_\Q(F,M)|=[F:\Q]$. Fix a
positive integer $n$ and an element $\sigma\in(\Z^n)^{\Hom(F,M)}$. Let
$\lambda$ be a place of $F$ with residue characteristic $l$ and
residue field $k_\lambda$, and let
$\rho_\lambda:G_F\to\GL_n(\Mbar_\lambda)$ be a continuous de Rham
representation. Then we say that $\rho$ has Hodge-Tate weights
$\sigma$ if for each embedding $\tau\in\Hom(F,M)$ inducing a place $v$
of $F$ via $F\into M\into M_\lambda$, the $\tau$-Hodge-Tate weights of
$\rho_\lambda|_{G_{F_v}}$ are $\sigma_\tau$. Similarly, we say that a
continuous representation $\rhobar_\lambda:G_F\to\GL_n(\kbar_\lambda)$
is Fontaine-Laffaille of weight $\sigma$ if $l$ is unramified in $F$,
and for each $v|l$, $\rhobar_\lambda|_{G_{F_v}}$ admits a crystalline
lift with Hodge-Tate weights $\sigma$ in the Fontaine--Laffaille range (equivalently, it has
Fontaine-Laffaille weights determined by $\sigma$ in the usual way).

\begin{lemma} \label{lemma:characters} Let $F/\Q$ be a number field,
  and let $\sigma \in \Z^{\Hom(F,M)}$ denote a fixed weight.  Suppose
  that there exist infinitely many primes $\lambda$ such that $G_F$
  admits a character: $\chi_{\lambda}: G_F \rightarrow
  \kbar_\lambda^\times$ with the following properties:
\begin{enumerate}
\item $\chi_{\lambda}$ is Fontaine-Laffaille of weight $\sigma$.
\item The conductor of $\chi_{\lambda}$ away from $l$ is bounded independently of $l$.
\end{enumerate}
Then, for infinitely many $\lambda$, there exists a finite order
character $\phi$ of $G_F$ such that $\chi_{\lambda} \phi^{-1} =
\overline{\psi_{\lambda}}$ for a fixed algebraic Hecke character
$\psi$ of weight $\sigma$. If $F$ does not contain a CM field, then
$\psi_{\lambda}$ is a power of the cyclotomic character times a finite
order character.
\end{lemma}

\begin{proof} The last sentence follows from the rest of the result by
  Weil's classification of algebraic Hecke characters. Note that since
  there are only finitely many finite order characters of $G_F$ with
  fixed ramification, it suffices to show that there is some algebraic
  Hecke character of weight $\sigma$. Without loss of generality we
  may assume that $M/\Q$ is Galois, and then by (the proof of) Weil's
  classification, it is enough to check that for any $g\in\Gal(M/\Q)$,
  $g\sigma\otimes 1$ annihilates $\cO_F^\times\otimes\R$. Replacing
  $\sigma$ by $g\sigma$ and each $\psi_\lambda$ by $g\psi_\lambda$, we
  see that it is enough to check that $\sigma\otimes 1$ annihilates
  $\cO_F^\times\otimes\R$.

  Regard each $\psi_\lambda$ as a character of $\A_F^\times/F^\times$ by
  class field theory. Since the ramification of the $\psi_\lambda$
  outside of $l$ is bounded independently of $l$, we see that there is
  a finite index subgroup $U$ of $\cO_F^\times$ such that each
  $\psi_\lambda|_U$ is just the composition of $\sigma$ and reduction
  mod $\lambda$. Thus, for any $u\in U$, we see that $(\sigma(u)-1)$ is
  divisible by $\lambda$; since this holds for infinitely many
  $\lambda$, we see that $\sigma(u)=1$. Since $U$ has finite index in
  $\cO_F^\times$, the result follows.\end{proof}

 \begin{corollary} \label{cor: dihedral even finiteness} Let $F$ be a
   totally real field.  Suppose that there are infinitely many primes
   $l$ and $2$-dimensional dihedral representations $\sbar_{\lambda}:
   G_F \rightarrow \GL_2(\kbar_\lambda)$ of fixed distinct
   Fontaine-Laffaille weights and fixed tame level. Then:
\begin{enumerate}
\item For all but finitely many $l$, $\sbar_{\lambda}$ is induced from
  a quadratic CM extension $F'/F$ unramified at $l$. In particular,
  $\sbar_\lambda$ is (totally) odd. (The field $F'$ may depend on $l$.)
\item For infinitely many $l$, $\sbar_{\lambda}$ is the reduction of
  the Galois representation associated to a fixed RAESDC automorphic representation
$\pi_{s}$ of $\GL_2(\A_F)$ which arises from the induction of an algebraic Hecke character 
for $\A^{\times}_{F'}$ for some CM extension $F'/F$.
\end{enumerate}
\end{corollary}
 \begin{proof}
   Each $\sbar_\lambda$ is induced from some quadratic extension
   $F_\lambda/F$ unramified outside of $l$ and a fixed set of
   places. By Lemma \ref{lem: Fontaine-Laffaille dihedral residual
     representations are unramified}, for all but finitely many
   $\lambda$, $F_\lambda/F$ is unramified at places dividing $l$. Thus
   there are only finitely many possible extensions $F_\lambda/F$, so
   it suffices to show that any extension $F'/F$ which occurs
   infinitely often is CM. However, if $F'$ is not CM, then it does
   not contain a CM subfield, and Lemma \ref{lemma:characters}
   (applied to the characters of $G_{F'}$ from which the $s_\lambda$
   are induced) contradicts the assumption that $s_\lambda$ has
   distinct Fontaine-Laffaille weights.

   The second part now follows from Lemma
   \ref{lemma:characters} in the same way.
\end{proof}
 
\begin{lemma}\label{lem: only finitely many exceptional GL_2} Let $F$
  be a number field. Fix a dimension $n$ and distinct
  Fontaine--Laffaille weights, and fix a bound on the order of the
  projective images of the Galois representations under consideration
  (for example, suppose that the projective images are all $A_4$,
  $S_4$ or $A_5$). Then there are only finitely many $\lambda$ such
  that there exists an irreducible representation
 $$\rhobar_{\lambda}: G_F \rightarrow \GL_n(\kbar_\lambda)$$
 such that $\rhobar_\lambda$ has these fixed distinct Fontaine-Laffaille
 weights, and has projective image of bounded order. \end{lemma}
\begin{proof}Suppose to the contrary that there are infinitely many
  such representations. By Fontaine-Laffaille theory, we see that the
  order of the projective image of $\rhobar_\lambda$ grows at least linearly
  with $l$; but it is also bounded, a contradiction.
\end{proof}

\begin{lemma}\label{lem: only finitely many reducible GL_2} Let $F$ be
 a totally real field. Then there are only finitely many $\lambda$ such
 that there exists an irreducible representation
$$\rho_{\lambda}: G_F \rightarrow \GL_2(\Mbar_\lambda)$$
such that $\det(\rho_\lambda(c_v))$ is independent of any infinite place $v$
of $F$, $\ad^0(\rho_\lambda)$ is crystalline with fixed distinct Hodge-Tate weights and fixed tame
level, and $\rhobar_{\lambda}$ is reducible.
\end{lemma}

\begin{proof} Assume not, so that infinitely many such representations
  exist. We may assume that $l$ is odd, unramified in $F$ and is
  sufficiently large that $\ad^0(\rho_\lambda)$ has Hodge-Tate weights
  in the Fontaine--Laffaille range. By Lemma~\ref{lemma:characters},
  we may deduce that, for infinitely many $\lambda$, there is an
  isomorphism $(\rhobar_{\lambda})^{\sss} \simeq \psi \omega^k \oplus
  1 \oplus \psi^{-1}\omega^{-k}$ for a fixed integer $k\ne 0$ and a
  fixed finite order character $\psi$. By Fontaine--Laffaille theory,
  we deduce that $\ad^0(\rho_\lambda)$ is ordinary at all primes
  $v|l$.

  Suppose that $\rhobar_\lambda\cong \psi\omega^k\phi\oplus
  \phi$. Applying Ribet's lemma, we obtain integral lattices in
  $\rho_\lambda$ which give nonzero classes in the groups
  $\Ext^1_{G_F}(\psi\omega^k\phi,\phi)$ and
  $\Ext^1_{G_F}(\phi,\psi\omega^k\phi)$. Consider the corresponding
  lattices in $\ad^0r_\lambda$. Looking at the ``top extension'' in
  $\ad^0(\rhobar_\lambda)$, we obtain classes in the
  groups $$H^1(F,\omega^k \psi), \qquad H^1(F,\omega^{-k}
  \psi^{-1}).$$ These classes are nonzero since $l\ne 2$, and are
  unramified outside $N$, $l$, and $\infty$ by construction. Moreover,
  in the second case, by the ordinarity of $\ad^0(\rho_\lambda)$ the
  class is also unramified and consequently trivial at all $v|l$.  If
  $M$ is a $G_{F}$-module, we may define a set of Selmer conditions as
  follows.  Let $\Sel = \{\Sel_v\}$ where $\Sel_v \subset H^1(G_v,M)$
  is defined to be:
\begin{enumerate}
\item $\Sel_v = H^1(G_v/I_v,M^{I_v})$ if $v \nmid l$.
\item $\Sel_v = 0$ if $v | l$.
\end{enumerate}
We note the following:
\begin{prop} Fix a place $v$, an integer $m \notin \{0,1\}$,
and a finite order character $\chi$.  Let $\omega$ denote the mod-$l$
cyclotomic character. Then
$H^1(F_v,\omega^m \cdot \chi) = 0$
for sufficiently large $l$.
\end{prop}
\begin{proof} We may assume that $v \nmid l$. Let $q = N(v)$ and let $p$ be the residue characteristic of $v$.
There is an equality
$$|H^1(F_v,M)| = |H^0(F_v,M)| |H^2(F_v,M)| =
|H^0(F_v,M)| |H^0(F_v,M^*)|.$$
Hence, if $H^1(F_v,\omega^m \cdot \chi)$ is non-trivial, then $\chi$ is unramified at $v$, and
$$\omega^m(\Frob_v) \chi(\Frob_v) \in \{1,q\}.$$
Since $\omega(\Frob_v) = q$, it follows that $\chi(\Frob_v) \in
\{q^{-m}, q^{1-m}\}$. If $\chi$ has order $d$, it follows that
$$(q^{dm} - 1)(q^{d(m-1)} - 1) \equiv 0 \pmod{l}.$$
Since $m \notin \{0,1\}$, this equality can only occur for finitely many $l$.
\end{proof}
It follows that for sufficiently large $l$ the classes constructed above lie in the Selmer groups
$H^1_{\Sel^*}(F,\omega^k \psi)$ and $H^1_{\Sel}(F,\omega^{-k} \psi^{-1})$ respectively,
where $\Sel^*$ is the dual Selmer condition (with no restriction on the class at $v|l$), with the
possible exception of the class in $H^1(F,\omega \psi)$ when $k = 1$.
We now consider three cases.
\begin{enumerate}
\item Suppose that $\omega^k \psi$  is (totally) odd. Then
the main conjecture for totally real fields as proven by Wiles~\cite{MR1053488} shows that (for $l$ odd) $l$ divides
$|H^1_{\Sel}(F,\omega^{-k} \psi^{-1})|$  if and only if
$l$ divides
$$L(0,\omega^{-k}   \psi^{-1}) \equiv L(-k,  \psi^{-1}) \ne 0.$$
\item Suppose that $\omega^k \psi$ is even and $k > 1$. Then, by
 Theorem 2.19 of \cite{MR1605752} (the global duality formula for
 Selmer groups, which is a reflection formula in this case), we deduce that
$$|H^1_{\Sel^*}(F,\omega^{k} \psi)| =  |H^1_{\Sel}(F,\omega^{1-k} \psi^{-1})|.$$
Once more by Wiles this group is non-trivial if and only if
$l$ divides
$$L(0,\omega^{1-k} \psi^{-1}) \equiv L(1-k, \psi^{-1}) \ne 0.$$
\item Suppose that $\omega^k \psi$ is even and $k = 1$. Then we still have a class
in $H^1_{\Sel}(F,\omega^{-1} \psi^{-1})$. Let $E = F(\psi)$.
Then, by restriction, we obtain a class in 
$$H^1_{\Sel}(F,\omega^{-1} \psi^{-1}) \hookrightarrow H^{1}_{\Sel}(E,\omega^{-1}) \hookrightarrow H^1_{\Sel}(E(\zeta_l),\F_l)^{\omega^{-1}}.$$
The latter group is isomorphic to the $\omega^{-1}$-part of the
$l$-torsion of the class group of $E(\zeta_l)$.  Yet, by Theorem~5.4
of~\cite{MR999397}, the $\omega^{-1}$ part of this group injects into
$K_2(\OL_E)/l$.  Since $K_2(\OL_E)$ is finite, this group is trivial
for $l$ sufficiently large.
\end{enumerate}
In each case, we deduce that $l$ divides a fixed non-zero rational number which
is independent of $l$, and hence $l$ is bounded.
\end{proof}

\section{Residual irreducibility for all but finitely many primes}

\subsection{}We now
bootstrap our previous arguments to prove the following result.
\begin{thm}\label{thm: all but finitely many irred over Q}Let $F$ be a
  totally real field. Suppose that $(\pi,\chi)$ is a
  RAESDC automorphic representation of $\GL_n(\A_F)$ with $n\le
  5$. Then all but finitely many of the residual representations
  $\rbar_{\lambda}(\pi)$ are irreducible. \end{thm}

We firstly establish the following corollary of Lemma~\ref{lem: 2
  divides 5} and Corollary \ref{cor: divisibility by a single character is impossible}.

\begin{lemma} \label{lemma:partitiontype} Let $F$ be a totally real field. Let $(\pi,\chi)$ be a RAESDC
  automorphic representation of $\GL_n(\A_F)$, $n\le 5$, and let
  $\cR=\{r_\lambda(\pi)\}$ be the associated weakly compatible
  system. Suppose that, for infinitely
  many primes $\lambda$, at least one of the following holds:
\begin{enumerate} 
\item $\rbar_{\lambda}(\pi)^{\sss}$ contains a character.
\item $\rbar_{\lambda}(\pi)^{\sss}$ contains a two dimensional 
dihedral representation.
\end{enumerate}
Then, respectively,  at least one of the following also holds:
\begin{enumerate}
\item $\cR$ is weakly divisible by a compatible system  of algebraic Hecke characters.
\item For some finite extension $E/F$, $\cR|_{G_E}$ is weakly divisibly by a
  direct sum of two compatible
systems of algebraic Hecke characters.
\end{enumerate}In particular, $\rbar_\lambda(\pi)$ is irreducible for
all but finitely many $\lambda$.
\end{lemma}

\begin{proof}
  Denote the corresponding  sub-representations of
  $\rbar_{\lambda}(\pi)^{\sss}$ by $\sbar_{\lambda}$. If the
  Hodge-Tate weights of $r_{\lambda}(\pi)$ are in the
  Fontaine--Laffaille range, then there are a fixed number of possible
  Fontaine-Laffaille weights of $\sbar_{\lambda}$, which are independent of
  $\lambda$. Similarly, there are finitely many possible Serre levels
  determined by the auxiliary ramification structure of $\cR$. 
  The result follows by 
  Lemma~\ref{lemma:characters} and Corollary~\ref{cor: dihedral even finiteness} 
  respectively (with the last sentence following from Lemma \ref{lem: 2 divides 5}).
\end{proof}

We now prove Theorem \ref{thm: all but finitely many irred
  over Q}. Assume for the sake of contradiction that there are
infinitely many $\lambda$ with $\rbar_\lambda(\pi)$ reducible. By
Lemma \ref{lemma:partitiontype}, there can only be finitely many
$\lambda$ such that $\rbar_\lambda(\pi)^\semis$ contains a character. This is
already a contradiction when $n\le 3$, and when $n=4$ or $n=5$ it
implies that there are infinitely many $\lambda$ for which
$\rbar_\lambda(\pi)^\semis\cong \sbar_\lambda\oplus \tbar_\lambda$ with
$\sbar_\lambda$ and $\tbar_\lambda$ both irreducible, $\dim
\sbar_\lambda=2$, and neither of $\sbar_\lambda$ and $\tbar_\lambda$
are dihedral.

\subsection{The case $n=5$} Since $5$ is odd, we see from Lemma
\ref{lem: odd-dimensional implies odd} that $r_\lambda(\pi)$ and thus
$\rbar_\lambda(\pi)$ factors
through $\GO_5$ with even multiplier. Since $2\ne 3$, we see that $\sbar_\lambda$
factors through $\GO_2$ with even multiplier, so $\sbar_\lambda$ is
dihedral. This can only happen for finitely many $\lambda$ by Lemma~\ref{lemma:partitiontype}.
\medskip

\subsection{The $n=4$ symplectic case} We argue as in
section~\ref{subsec:symplectic char 0}. Let $\cR$ be the compatible system of Galois representations associated to $(\pi,\chi)$,
and define $\cA:= \wedge^2(\cR) - \chi$, a compatible system of Galois representations such that
$a_{\lambda}:G_F \rightarrow \GL_5(\Qbar_l)$ has image in $\GO_5(\Qbar_l)$. Again, this
compatible system is odd and regular.

Since $\rbar_\lambda(\pi)=\sbar_\lambda\oplus \tbar_\lambda$, we have
$$\abar_{\lambda} \oplus \chibar_{\lambda} = \sbar_{\lambda} \otimes \tbar_{\lambda} \oplus 
\det(\sbar_{\lambda}) \oplus \det(\tbar_{\lambda}).$$ In particular,
as there are two characters on the right hand side, the representation
$\abar_\lambda$ contains a character for infinitely many $\lambda$,
and as in the proof of Lemma~\ref{lemma:partitiontype}, we deduce that
the compatible system $\cA$ is weakly divisible by the compatible system of a
character. 

Arguing as in section~\ref{subsec:symplectic char 0}, we deduce that
there is a finite Galois extension $F'/F$ of totally real fields and a
RAESDC automorphic representation $\pi'$ of $\GL_2(\A_{F'})$, which is
not of CM type, such that, for all $\lambda$, we have
$a_{\lambda}|_{G_{F'}}\cong\Sym^4r_{\lambda}(\pi')\otimes\det(r_{\lambda}(\pi'))^{-2}$. Then
for all but finitely many $\lambda$ the projective image of
$\rbar_{\lambda}(\pi')$ contains $\PSL_2(\F_l)$ and
$\abar_\lambda|_{G_{F'}}$ is irreducible, a contradiction.

\subsection{The $n=4$ orthogonal case}

It will be useful in the sequel to 
exploit the exceptional isomorphism of Lie groups $\sof_4 \simeq \slf_2 \times \slf_2$.
More precisely:

\begin{lem}
  \label{lem: GO4 implies Asai or tensor}Let $F$ be a number field. Suppose that
  $r:G_F\to\GO_4(\Qlbar)$ is a continuous representation. Then either:
  \begin{enumerate}
  \item there are continuous representations $a$,
    $b:G_F\to\GL_2(\Qlbar)$ with $r\cong a\otimes b$, or
  \item there is a quadratic extension $K/F$ and a continuous
    representation $a:G_K\to\GL_2(\Qlbar)$ with $r|_{G_K}\cong
    a\otimes a^c$, where $\Gal(K/F)=\{1,c\}$.
  \end{enumerate}

\end{lem}
\begin{proof}
  We have an exact sequence
  \[0\to\Qlbar^\times\to\GL_2(\Qlbar)\times\GL_2(\Qlbar)\to\GO_4(\Qlbar)\to\{\pm
  1\}\to 0\] (cf. section 1 of~\cite{MR1874921}). Suppose firstly that
  the composite $r:G_F\to\GO_4(\Qlbar)\to\{\pm 1\}$ is not
  surjective. Then the obstruction to lifting $r$ to a homomorphism
  $G_F\to\GL_2(\Qlbar)\times\GL_2(\Qlbar)$ lies in
  $H^2(G_F,\Qlbar^\times)$, which vanishes by the proof of Theorem
  5.4 of~\cite{MR2461903}. If the  composite $r:G_F\to\GO_4(\Qlbar)\to\{\pm 1\}$ is 
  surjective, then we let $G_K$ be the kernel of this composite, and
  the result follows as in the previous case.
\end{proof}

By Lemma \ref{lem:
  GO4 implies Asai or tensor}, we may assume either that for infinitely
many $\lambda$ we have $r_\lambda\cong a_\lambda\otimes b_\lambda$ for
some $a_\lambda$, $b_\lambda:G_F \to\GL_2(\Qlbar)$, or that for
infinitely many $\lambda$ there is a quadratic extension $K=K_\lambda/F$
with $r|_{G_K}\cong a_\lambda\otimes a_\lambda^c$ for some
$a_\lambda:G_K\to\GL_2(\Qlbar)$.

Suppose that we are in the first case. 
If $\abar_{\lambda}$ and $\bbar_{\lambda}$ are both reducible, then the
semi-simplification of $\abar_{\lambda} \otimes \bbar_{\lambda}$ consists of four characters, contrary to assumption.
Hence, without loss of generality, we may assume that $\bbar_{\lambda}$ is irreducible infinitely often.
There is an isomorphism
$$\wedge^2 r_{\lambda}(\pi) \otimes (\det r_{\lambda}(\pi))^{-1} \simeq
\ad^0 a_{\lambda} \oplus \ad^0 b_{\lambda},$$ from which we see that
$\ad^0 a_\lambda$ and $\ad^0 b_\lambda$ are both de Rham, and are
crystalline for all but finitely many $\lambda$. They
are both regular (this can be seen from a consideration of the
Hodge-Sen-Tate weights of $a_\lambda$ and $b_\lambda$).

\medskip

 After passing to a subset, we may assume that the
 Hodge-Tate weights of $\ad^0(a_{\lambda})$ and $\ad^0(b_{\lambda})$
 are independent of $\lambda$. Assume firstly that $\abar_\lambda$ and
 $\bbar_\lambda$ are irreducible for infinitely many $\lambda$.

 By Lemma \ref{lem: only finitely many exceptional GL_2} and the
 classification of finite subgroups of $\PGL_2(\Flbar)$ (cf. Theorem
 2.47 of \cite{MR1605752}), we may assume that $\abar_\lambda$ and
 $\bbar_\lambda$ are either dihedral or have images containing
 $\SL_2(\Fl)$.

 If both $\abar_\lambda$ and $\bbar_\lambda$ have image containing
 $\SL_2(\Fl)$, then $\abar_\lambda\otimes\bbar_\lambda$ is either
 irreducible or breaks up as a sum of a character and a 3-dimensional
 representation, a contradiction. If without loss of generality
 $\abar_\lambda$ is dihedral and $\bbar_\lambda$ has image containing
 $\SL_2(\Fl)$, let $K/F$ be the quadratic extension from which
 $\abar_\lambda$ is induced. Then $\bbar_\lambda|_{G_K}$ is
 irreducible, so $\abar_\lambda\otimes\bbar_\lambda$ is irreducible, a
 contradiction. The remaining case is that $\abar_\lambda$ and
 $\bbar_\lambda$ are both dihedral. Then $\rbar_\lambda(\pi)$ is
 completely decomposable over some quartic extension, which implies
 that $\sbar_\lambda$ and $\tbar_\lambda$ are both dihedral, a contradiction.

 We may thus assume that for infinitely many $\lambda$,
 $\abar_\lambda$ is reducible and $\bbar_\lambda$ is irreducible. If
 $\bbar_\lambda$ is dihedral then one of $\sbar_\lambda$ and
 $\tbar_\lambda$ is dihedral, which can only occur finitely often, so by Lemma \ref{lem: only finitely many
   exceptional GL_2} (applied to $\ad^0 b_\lambda$) we may assume that
 the image of $\bbar_\lambda$ contains $\SL_2(\Fl)$. Then for
 $\lambda$ sufficiently large $\ad^0\bbar_\lambda$ is irreducible, so
 as in the proof of Proposition \ref{prop: 2d representations are odd}
 we see that $\ad^0b_\lambda$ is potentially automorphic and
 $b_\lambda$ is odd. Since the multiplier character of
 $r_\lambda(\pi)=a_\lambda\otimes b_\lambda$ is even, we see that $\det
 a_\lambda(c_v)$ is independent of $v|\infty$. Then Lemma \ref{lem: only finitely many reducible GL_2} implies that there are only
 finitely many $\lambda$ for which such an $a_\lambda$ can exist.

This contradiction means that we may assume that we are in the second
case, so that for infinitely many $\lambda$, there is a quadratic extension $K/F$ (which might depend on
$\lambda$) and a continuous representation
$a_\lambda:G_K\to\GL_2(\Qlbar)$ such that $r_\lambda(\pi)|_{G_K}\cong
a_\lambda\otimes a_\lambda^c$.

We claim that $\abar_\lambda$ is necessarily dihedral for all but
finitely many of the $\lambda$ under consideration. We prove this by
eliminating the other possibilities. Firstly, $\abar_\lambda$ cannot
be reducible, because if $\abar_\lambda\cong\phibar\oplus \chibar$,
then $$\abar_{\lambda} \otimes \abar^c_{\lambda}\cong\phibar \otimes
\phibar^c \oplus \chibar \otimes \chibar^c \oplus \phibar \otimes
\chibar^c \oplus \chibar \otimes \phibar^c,$$ and the first two
characters descend to $\Q$, so that $\rbar_\lambda(\pi)$ would have
one-dimensional subrepresentations.

If $\abar_\lambda$ has projective image $A_4$, $S_4$ or $A_5$, then the
projective image of $\rbar_\lambda(\pi)$ is bounded independently of
$\lambda$. By Lemma \ref{lem: only finitely many exceptional
  GL_2}, this can only happen for finitely many $\lambda$.

The image of $\abar_{\lambda}$ cannot contain $\SL_2(\F_{\ell})$. If
it did, then $\abar_{\lambda} \otimes \abar^{c}_{\lambda}$ would
either be irreducible or a sum of an irreducible three-dimensional
representation and a character, depending on whether the projective
representation associated to $\abar_{\lambda}$ extends to $F$ or
not. This again contradicts the assumption that $\rbar_\lambda(\pi)$ is a
sum of two irreducible 2-dimensional representations.

Having eliminated the other possibilities, we see that for infinitely
many $\lambda$, $\abar_\lambda$ is dihedral. Then for infinitely many $\lambda$, $\sbar_\lambda$ and
$\tbar_\lambda$ each become reducible over quartic extensions of $F$,
and are thus dihedral. This contradiction completes the proof.

\section{Lie algebras}In this section we prove that the Lie algebras
of the images of the $r_\lambda(\pi)$ are independent of
$\lambda$. More specifically, we prove the following theorem.
\begin{thm}
  \label{thm: Lie algebras are independent of l}Let $F$ be a totally
  real field. Suppose that
  $(\pi,\chi)$ is a RAESDC automorphic representation of $\GL_n(\A_F)$
  with $n\le 5$. Let $G_\lambda$ denote the Zariski closure of the
  image $r_\lambda(\pi)(G_F)$, and let $\fg_\lambda=\gz_\lambda\oplus\gh_\lambda$ denote the Lie
  algebra of $G_\lambda$, where $\gz_\lambda$ is abelian and
  $\gh_\lambda$ is semisimple. Then $\fg_\lambda$ is independent of
  $\lambda$, and $\gh_\lambda$ is either $\sl_2$, $\so_4$ or $\sp_4$ (if
  $n=4$), or $\so_5=\sp_4$ (if $n=5$).
\end{thm}
\begin{proof}
  The result is trivial if $n=1$ and standard if $n=2$, so we may
  assume that $n\ge 3$. If $(\pi,\chi)$ is the automorphic induction
  of a character, then certainly $\fg_\lambda$ is abelian and
  independent of $\lambda$. If $n=4$ and $(\pi,\chi)$ is the
  automorphic induction of a RAESDC automorphic representation of
  $\GL_2(\A_{F'})$ for $F'/F$ a quadratic totally real extension, then
  either this representation is of CM type, and $(\pi,\chi)$ is the
  automorphic induction of a character, or $\fg_\lambda$ is
  independent of $\lambda$ and $\gh_\lambda$ is equal to $\sl_2$
  (acting reducibly).

  Excluding these cases, by Corollary \ref{cor:induct} (and its proof)
  we may assume that $(\pi,\chi)$ is not an automorphic induction, and
  that $r_\lambda(\pi)$ is strongly irreducible. In general the
  independence of $\gz_\lambda$ of $\lambda$ is an easy consequence of
  Schur's lemma. Therefore we need only determine
  $\gh_\lambda$. Suppose firstly that the compatible system
  $\{r_\lambda(\pi)\}$ is weakly divisible by a compatible system of
  algebraic Hecke characters. Then by Lemma \ref{lem: 2 divides 5} we
  see that $n=3$ or $5$, and that $\gh_\lambda=\sl_2$, acting through
  the $(n-1)$st symmetric power representation, independently of
  $\lambda$.

  Conversely, if $n=3$ or $5$ and for some $\lambda$ we have
  $\gh_\lambda=\sl_2$ acting through the $(n-1)$st symmetric power
  representation, then we claim that the compatible system
  $\{r_\lambda(\pi)\}$ is weakly divisible by a compatible system of
  algebraic Hecke characters. To see this, write $G$ for the Zariski
  closure of $r_\lambda(\pi)(G_F)$, and $G^0$ for the connected
  component of the identity. Then the derived subgroup of $G^0$ must
  be $\PSL_2$, and since $\PSL_2$ has no outer automorphisms, Schur's
  lemma shows that $G$ is necessarily of the form
  $Z(G)\times\PSL_2$. Since $Z(G)$ acts via a character (again by
  Schur's lemma), we see that the compatible system
  $\{r_\lambda(\pi)\}$ is weakly divisible by a compatible system of
  algebraic Hecke characters, as required.

Examining the table in Proposition \ref{prop: list of algebras}, we
see that we are done unless $n=4$. In this case if $\chi$ is odd then
each $r_\lambda(\pi)$ has even multiplier and is thus orthogonal, and
we see from the same table that $\gh_\lambda=\so_4$ for all
$\lambda$. If $\chi$ is even then for each $\lambda$ either
$\gh_\lambda=\sl_2$ (acting via $\Sym^3$) or $\gh_\lambda=\sp_4$. We
distinguish between these two possibilities by arguing as in Section
\ref{subsec:symplectic char 0}. Consider the compatible system $\cA:=
\wedge^2(\cR) - \chi$. This is a compatible system of odd, regular
Galois representations such that $a_{\lambda}:G_F \rightarrow
\GL_5(\Qbar_l)$ has image in $\GO_5(\Qbar_l)$. If for some $\lambda$
we have $\gh_\lambda=\sl_2$ then as above the compatible system $\cA$
is weakly divisible by a compatible system of algebraic characters,
and the argument of the proof of Lemma \ref{lem: 2 divides 5} shows
that $\gh_\lambda=\sl_2$ for all $\lambda$, as required.
\end{proof}

\section{Non self-dual representations of \texorpdfstring{$\GL_3$}{GL(3)} and 
\texorpdfstring{$\GL_4$}{GL(4)}}
In this
section, we follow \cite{ramakrishnan2009irreducibility} and sketch a
proof that our earlier irreducibility results extend to the case of
regular algebraic cuspidal automorphic representations $\pi$ of
$\GL_3(\A_F)$ or $\GL_4(\A_F)$, $F$ a totally real field, without assuming that $\pi$ is
essentially self-dual, but with the assumption that the Galois
representations $r_\lambda(\pi)$ exist.

Assume throughout this section that $F$ is totally real, that $n=3$ or $4$ and that there is a weakly compatible
system $\{r_\lambda(\pi)\}$ of Galois representations associated to
$\pi$. We will demonstrate the required irreducibility results by
reducing to the essentially self-dual case.

Note firstly that the proof of Corollary \ref{cor:induct} made no use
of the essential self-duality of $\pi$, so we have the following.

\begin{lemma}\label{lem: not strongly irred}  Let
  $\lambda$ be a prime such that $r_{\lambda}(\pi)$ is irreducible.
Then either:
\begin{enumerate} 
\item $r_\lambda(\pi)$ is strongly irreducible, or
\item $\pi$ is an automorphic induction,  $r_{\lambda}(\pi)$ is irreducible for all $\lambda$, and
$\rbar_{\lambda}(\pi)$ is irreducible for all but finitely many $\lambda$.
\end{enumerate}
\end{lemma}
The following Lemma and Corollary will be our main tool to reduce to
the essentially self-dual case.

\begin{lemma}\label{lem: wedge squared implies self dual representation} Suppose that $r:G_{F} \rightarrow \GL_4(\Qbar_l)$ is
  strongly irreducible, and suppose that $\wedge^2 r:G_{F} \rightarrow
  \GL_6(\Qbar_l)$ is not strongly irreducible. Then $r \simeq r^{\vee}
  \chi$ for some character $\chi$.
\end{lemma}

\begin{proof} This is a standard argument (cf. Theorem 6.5 of
  \cite{0712.4315}). Consider the Zariski closure $G$ of the image of
  $r$. Let $G^0$ denote the connected component of $G$, let $\gog$ be
  the Lie algebra of $G^0$, and write $\gog = \gz \oplus \gh$, with
  $\gz$ is abelian and $\gh$ semisimple.  By assumption, $G^0$ acts
  irreducibly in dimension $4$.  If $\wedge^2 r$ is not strongly
  irreducible, then $\gh = \slf_2$, $\sof_4$, or $\spf_4$. It follows
  that $G^0$ preserves a symplectic or orthogonal form, from which it
  is easy to deduce (using, for example, the facts that the normalizers of $\Sp_4$
  and $\SO_4$ in $\GL_4$ are respectively $\GSp_4$ and $\GO_4$) that
  the image of $r$ is symplectic or orthogonal, as required.
\end{proof}
\begin{cor}\label{cor: wedge squared irreducible}
  Suppose that $n=4$, that $\pi$ is not essentially self-dual, and that for some
  $\lambda$, $r_\lambda(\pi)$ is strongly irreducible. Then $\wedge^2
  r_\lambda(\pi)$ is strongly irreducible.
\end{cor}
\begin{proof}Suppose that $\wedge^2 r_\lambda(\pi)$ is not strongly
  irreducible. By Lemma \ref{lem: wedge squared implies self dual
    representation}, we see that $r_\lambda(\pi)\simeq
  r_\lambda(\pi)^\vee\chi_\lambda$ for some character $\chi_\lambda$;
  but then strong multiplicity one for $\GL_4$ implies that $\pi$ is
  essentially self-dual, a contradiction.
\end{proof}
\begin{lem}
  \label{lem: 3+1 doesn't happen}If $n=4$, then it is impossible for
  $r_\lambda(\pi)$ to have a 1-dimensional summand.
\end{lem}
\begin{proof}
  This may be proved in exactly the same way as Proposition 7.8 of
  \cite{ramakrishnan2009irreducibility}. Suppose that
  $r_\lambda(\pi)=\chi_\lambda\oplus s_\lambda$ with $\chi_\lambda$ a
  character. Then $\chi_\lambda$ and $\det s_\lambda$ are both
  algebraic characters of $G_F$, so arise from automorphic
  representations $\chi$ and $\nu$ of $\GL_1(\A_F)$. One easily
  obtains an equality of incomplete
  $L$-functions \[L^S(s,\pi\otimes\chi^{-1})L^S(s,\pi^\vee\otimes\nu\chi^{-1})=L^S(s,\wedge^2(\pi)\otimes\chi^{-1})\zeta^S(s)L^S(s,\nu\chi^{-3}).\]Since
  $\wedge^2(\pi)$ is an isobaric automorphic representation of
  $\GL_6(\A_F)$ by \cite{MR1937203}, we see that the left hand side is
  holomorphic at $s=1$, but the right hand side has at least a simple
  pole at $s=1$, a contradiction.
\end{proof}
\begin{lem}
  \label{lem: irreducible density one}For a density one set of primes
  $\lambda$, $r_\lambda(\pi)$ is irreducible.
\end{lem}
\begin{proof}
  Suppose not. By Lemma \ref{lem: 3+1 doesn't happen}, there is a set
  of primes $\lambda$ of positive density such that $r_\lambda(\pi)$
  decomposes as a sum of irreducible representations of dimension at
  most $2$. The result now follows by the same proof as Theorem
  \ref{thm: irreducibility for density one}.
\end{proof}
\begin{thm}\label{thm: non self-dual irred char 0}
  $r_\lambda(\pi)$ is irreducible for all $\lambda$. 
\end{thm}
\begin{proof}
  Let $\cR$ denote the compatible system $\{r_\lambda(\pi)\}$. By
  Theorem \ref{thm: all irred char 0}, we may assume that $\pi$ is not
  essentially self-dual. By Lemmas \ref{lem: not strongly irred} and
  \ref{lem: irreducible density one}, and Proposition 5.2.2 of
  \cite{blggt}, we may assume that $r_\lambda(\pi)$ is strongly
  irreducible for some $\lambda$. Then the proof of Lemma \ref{lem: 2
    divides 5} goes through verbatim, as does that of Corollary
  \ref{cor: divisibility by a single character is impossible}, and we
  see that it is impossible for $r_{\lambda'}(\pi)$ to have a
  one-dimensional summand for any $\lambda'$.

  We are done if $n=3$. If $n=4$, the only possibility is that for
  some $\lambda'$, $r_{\lambda'}(\pi)\cong s_{\lambda'}\oplus
  t_{\lambda'}$, with $s_\lambda$ and $t_\lambda$ both
  2-dimensional. Since we have assumed that $\pi$ is not essentially
  self-dual, we see from Corollary \ref{cor: wedge squared
    irreducible} that (using the same $\lambda$ as in the first
  paragraph) $\wedge^2 r_\lambda(\pi)$ is strongly irreducible. On the
  other hand, \[\wedge^2 r_{\lambda'}(\pi)\cong s_{\lambda'}\otimes
  t_{\lambda'}\oplus\det(s_{\lambda'})\oplus\det(t_{\lambda'}).\] It
  follows that the compatible system $\wedge^2\cR$ is weakly divisible
  by the compatible system of a character (in fact, two
  characters). After twisting, we may suppose that this character is
  trivial. We then obtain a contradiction as in the proof of Theorem
  \ref{thm: punt} (note that the semisimple part of the Lie algebra of
  the Zariski closure of $\wedge^2 r_\lambda(\pi)(G_F)$ is $\sl_4$
  acting via $\wedge^2$, and it is not the case that every element of
  $\sl_4$ fixes some vector under this action).
\end{proof}

\begin{thm}
  For all but finitely many $\lambda$, $\rbar_\lambda(\pi)$ is irreducible.
\end{thm}
\begin{proof}
  By Theorem \ref{thm: all but finitely many irred over Q}, it is enough to assume that $\pi$ is not essentially
  self-dual. Again, the proof of Lemma \ref{lemma:partitiontype} goes
  over without change to the present setting, completing the proof if
  $n=3$. If $n=4$, it suffices to show that there cannot be infinitely
  many $\lambda$ for which $\rbar_\lambda(\pi)=\sbar_\lambda\oplus
  \tbar_\lambda$ with $s_\lambda$ and $t_\lambda$ 2-dimensional. However,
  we again note that in this case we have \[\wedge^2 r_\lambda(\pi)=
  \sbar_{\lambda} \otimes \tbar_{\lambda} \oplus \det(\sbar_{\lambda})
  \oplus \det(\tbar_{\lambda}),\] and we deduce that the compatible
  system $\{\wedge^2 r_\lambda(\pi)\}$ is weakly divisible by the
  compatible system of a character. This gives a contradiction as in
  the proof of Theorem \ref{thm: non self-dual irred char 0}.
\end{proof}

\begin{thm}
  Let $\fg_\lambda$ be the Lie algebra of the Zariski closure of
  $r_\lambda(G_F)$. Then $\fg_\lambda$ is independent of
  $\lambda$.
\end{thm}
\begin{proof}
  As in the proof of theorem Theorem \ref{thm: Lie algebras are
    independent of l}, if we write
  $\fg_\lambda=\gh_\lambda\oplus\gz_\lambda$ with $\gh_\lambda$
  semisimple and $\gz_\lambda$ abelian, it suffices to show that
  $\gh_\lambda$ is independent of $\lambda$.

 If $n=3$, then the proof
  of Theorem \ref{thm: Lie algebras are independent of l} goes through
  unchanged, so we may assume that $n=4$. By Theorem \ref{thm: Lie
    algebras are independent of l}, we may assume that $\pi$ is not
  essentially self-dual, and by Lemma \ref{lem: not strongly irred}
  and Theorem \ref{thm: non self-dual irred char 0}, we may assume that
  $r_\lambda(\pi)$ is strongly irreducible for all $\lambda$. It then
  follows from the proofs of Lemma \ref{lem: wedge squared implies
    self dual representation} and Corollary \ref{cor: wedge squared
    irreducible} that $\gh_\lambda=\sl_4$ for all $\lambda$, as
  required.
\end{proof}
\bibliographystyle{amsalpha}
\bibliography{irreducibility}
\end{document}